\documentclass{amsart}
\usepackage{xypic}
\usepackage{graphicx, amsmath, amssymb, amsthm}
\usepackage{mathabx}
\usepackage{tikz-cd}
\usepackage{hyperref}

\font\maljapanese=dmjhira at 2ex 
\def\yo{\textrm{\maljapanese\char"48}} 

\newcommand{\support}[1]{The author was supported by {#1}.}

\newcommand{\NSFFour}{NSF Grant DMS--1811189}

\newcommand{\MAHAddress}{University of California Los Angeles, Los Angeles, CA 90095}
\newcommand{\MAHemail}{\tt{mikehill@math.ucla.edu}}

\newcommand{\Z}{{\mathbb  Z}}

\newcommand{\R}{{\mathbb R}}
\newcommand{\F}{{\mathbb F}}
\DeclareMathOperator{\MU}{MU}
\newcommand{\MUR}{\MU_{\R}}

\newcommand{\BUR}{BU_{\R}}

\newcommand{\smashover}[1]{\underset{#1}{\wedge}}
\newcommand{\timesover}[1]{\underset{#1}{\times}}
\newcommand{\boxover}[1]{\underset{#1}{\Box}}
\newcommand{\tensorover}[1]{\underset{#1}{\otimes}}
\newcommand{\otimesover}[1]{\tensorover{#1}}

\DeclareMathOperator{\RO}{RO}

\newcommand{\fpr}{{\_}r}
\newcommand{\coinduction}{coinduction}

\newcommand{\Ind}{\big\uparrow} 
\newcommand{\ind}{\uparrow}
\DeclareMathOperator{\Hom}{Hom}
\DeclareMathOperator{\Ext}{Ext}
\DeclareMathOperator{\Tor}{Tor}

\DeclareMathOperator{\Map}{Map}

\DeclareMathOperator{\Stab}{Stab}

\newcommand{\res}{res}
\newcommand{\tr}{tr}

\newcommand{\m}[1]{{\protect\underline{#1}}}
\newcommand{\mZ}{\m{\Z}}
\newcommand{\mM}{\m{M}}

\newcommand{\mR}{\m{R}}

\newcommand{\mA}{\m{A}}

\newcommand{\mstar}{\m{\star}}

\newcommand{\cc}[1]{\mathcal #1}

\newcommand{\cA}{\cc{A}}

\newcommand{\Sp}{\mathcal Sp}

\newcommand{\Set}{\mathcal Set}

\newcommand{\Top}{\mathcal Top}

\newcommand{\Mod}{\mathcal Mod}
\newcommand{\RMod}{R\mhyphen\Mod}

\newcommand{\Fin}{\mathcal Fin}

\mathchardef\mhyphen=45

\newcommand{\EM}{Eilenberg--Mac~Lane}
\newcommand{\DL}{Dyer--Lashof}
\newcommand{\Boxover}[1]{\underset{#1}{\Box}}

\numberwithin{equation}{section}

\newtheorem{theorem}{Theorem}[section]

\newtheorem{corollary}[theorem]{Corollary}

\newtheorem{proposition}[theorem]{Proposition}

\newtheorem*{theorem*}{Theorem}
\newtheorem*{proposition*}{Proposition}
\newtheorem{example}[theorem]{Example}

\theoremstyle{remark}
\newtheorem{remark}[theorem]{Remark}

\newtheorem{notation}[theorem]{Notation}

\theoremstyle{definition}

\newtheorem{definition}[theorem]{Definition}
\newtheorem*{definition*}{Definition}

\newcommand{\defemph}[1]{\textbf{#1}}

\title{Freeness and equivariant stable homotopy}

\author{Michael A.~Hill}
\thanks{\support{\NSFFour}}
\address{\MAHAddress}
\email{\MAHemail}

\begin{document}

\begin{abstract}
We introduce a notion of freeness for \(RO\)-graded equivariant generalized homology theories, considering spaces or spectra \(E\) such that the \(R\)-homology of \(E\) splits as a wedge of the \(R\)-homology of induced virtual representation spheres. The full subcategory of these spectra is closed under all of the basic equivariant operations, and this greatly simplifies computation. Many examples of spectra and homology theories are included along the way.

We refine this to a collection of spectra analogous to the pure and isotropic spectra considered by Hill--Hopkins--Ravenel. For these spectra, the \(RO\)-graded Bredon homology is extremely easy to compute, and if these spaces have additional structure, then this can also be easily determined. In particular, the homology of a space with this property naturally has the structure of a co-Tambara functor (and compatibly with any additional product structure). We work this out in the example of \(BU_{\mathbb R}\) and coinduced versions of this.

We finish by describing a readily computable bar and twisted bar spectra sequence, giving Bredon homology for various \(E_{\infty}\) pushouts, and we apply this to describe the homology of \(BBU_{\mathbb R}\).
\end{abstract}

\maketitle

\section{Introduction}
Equivariant cohomology is often viewed as very difficult to compute. In full generality, this is often true, as many computations which non-equivariantly were completed in the 1950s and 1960s are still out of reach. Addtionally, the kinds of cellular decompositions which geometrically arise are often not adapted to easy computation, further compounding the problem. Many computations in the literature require significant amounts of hard work, even for ordinary (Bredon) homology (see, for example, the recent papers of Dugger on equivariant Grassmanians \cite{DuggerGr} and Hazel on \(C_{2}\)-surfaces \cite{HazelSurfaces}). 

In this paper, we build on a class of spectra introduced by Ferland--Lewis \cite{FerLew}, focusing on a certain subcategory of spaces and spectra for which essentially all of these problems go away. The basic definition is motivated by algebra.
\begin{definition*}
Let \(R\) be an \(E_\infty\)-monoid in genuine \(G\)-spectra. A \(G\)-spectrum \(E\) has {\defemph{\(R\)-free homology}} if \(R\wedge E\) splits as a wedge of \(R\)-modules of the form
\[
R\wedge (G_+\smashover{H}S^V),
\]
where \(V\) is a virtual representation of \(H\).
\end{definition*}
These classes of spectra contain many geometrically meaningful spaces and spectra. Delightfully, these \(R\)-free spectra are closed under most of the usual operations in equivariant homotopy.
\begin{theorem*}
If \(R\) is an \(E_\infty\)-monoid in genuine \(G\)-spectra, then the category of \(R\)-free spectra is closed under 
\begin{enumerate}
\item coproducts, 
\item restriction along arbitrary homomorphisms, 
\item induction from a subgroup,
\item the smash product, and 
\item norm maps, if \(R\) is actually a \(G\)-\(E_\infty\) ring spectrum.
\end{enumerate}
\end{theorem*}

For Bredon homology, this gives a large class of spaces and spectra for which the cohomology is easy to describe with almost arbitrary coefficients. Most excitingly, it means we can describe the full coalgebra (in fact, co-Tambara functor) structure on the homology of these spaces and on the cohomology of equivariant commutative monoid objects. Closure under norms here gives a formula for the Bredon homology of coinduced spaces with various coefficients, which in turn gives ways to understand Bredon homology and cohomology of certain {\EM} spaces.

After describing a host of examples, we restrict focus to a class of spectra for which everything is described by the underlying homology. The slice filtration of \cite{HHR} gives a version of the Postnikov tower where we use various representation spheres instead of ordinary spheres. In the nicest cases, such as those built out of the norms of the Fujii--Landweber spectrum of Real bordism \(\MUR\), the slice associated graded is a wedge of regular representation spheres smashed with computationally tractible Eilenberg--Mac Lane spectra \cite{Fujii, LandweberMUR} (see also \cite{HuKriz}). 

We consider \(H\mZ\)-free spectra where the [induced] virtual representation spheres are only in regular representation dimensions. These assumptions allow us to reduce almost any computational question to a question about the non-equivariant homotopy, tying things to classically studied objects. We demonstrate the efficiency of this by giving the full Tambara and co-Tambara functor structures on the homology of \(\BUR\) and of \(\Map^{C_{2}}(G,\BUR)\). We also describe the action of the \(C_{2}\)-{\DL} algebra on the mod \(2\)-homology of \(\BUR\).

We close with applications to the bar/Rothenberg--Steenrod and Eilenberg--Moore spectral sequences. When the spaces in question are \(R\)-free, the \(E_{2}\)-terms of the usual spectral sequences have the expected form, and we use this to compute the homology of \(B\BUR\) and of the coinduced space \(\Map^{C_{2}}(G,B\BUR)\) for all finite \(G\). As an aside, we also mention the sign-twisted analogues of these classical spectral sequences when \(G=C_{2}\), giving ways to compute the homology of the signed bar construction or the cohomology of the twisted homotopy pullback and signed loop spaces.

Throughout the paper, our emphasis is on the conceptual understanding of the objects and on explicit examples. We include many examples of spaces and spectra of interest, showing how they fit into this framework, working to demystify equivariant computations.

\subsection*{Conventions and notation}
In all that follows, we work in ``genuine'' \(G\)-spectra for a finite group \(G\). Much of what we say will actually be model agnostic; we will largely talk about results in the homotopy category. When discussing the difference between \(E_\infty\) and \(G\)-\(E_\infty\) monoids, however, we will implicitly be working in equivariant orthogonal or symmetric spectra, since both have well-developed notions of the norm \cite{HHR}, \cite{Hausmann}.

\section{\texorpdfstring{\(\RO\)}{RO}-graded homology}
Many of the spaces which arise geometrically can be built not out of cells of the form ``disk in a [virtual] representation \(V\)'' but rather out of more general cells of the form
\[
G_+\smashover{H} D(V),
\]
where \(V\) is a [virtual] \(H\)-representation. Algebraic constructions like the norm automatically build in this more general kind of \(\RO\)-grading, considering instead objects graded on pairs consisting of a subgroup \(H\) and a virtual representation of \(H\). A more coordinate free version is given by considering Thom spectra of virtual bundles over finite \(G\)-sets; a particular model of this is the work of Angeltveit--Bohmann \cite{AngeltveitBohmann}.

\begin{definition}[{\cite[Definition 2.7]{HHRkH}}]
If \(T\) is a finite \(G\)-set and \(V\) is an equivariant virtual  bundle over \(T\), then let \(M(V)\) be the Thom spectrum of \(V\) and 
\[
\pi_{\mstar}(E)\big(T,V\big)=\big[M(V),E\big]^{G}.
\]
\end{definition}

\begin{remark}
If \(T\) is a transitive \(G\)-set, then a choice of point \(t\in T\) gives an equivariant equivalence
\[
T\cong G/\Stab(t),
\]
and restriction to \(t\) gives an equivalence of categories between \(\Stab(t)\)-equivariant virtual representations and virtual equivariant vector bundles over \(T\). 
\end{remark}

\begin{notation}
In the case \(T=G/H\), so \(V\) gives a virtual \(H\)-representation \(V_{H}\), let
\[
E^{H}_{V_{H}}(S^{0})=\pi_{\mstar}(E)(T,V).
\]
\end{notation}

These abelian groups assemble into a kind of Mackey functor, twisted by these bundles. This generalizes the earlier work of Ferland--Lewis \cite{FerLew}.

\begin{proposition}
If \(f\colon S\to T\) is a map of finite \(G\)-sets and if \(V\to T\) is a virtual equivariant bundle, then \(f\) induces a transfer map
\[
\pi_{\mstar}(E)\big(S,f^{\ast}V\big)\xrightarrow{T_{f}} \pi_{\mstar}(E)\big(T,V\big)
\]
and a restriction map
\[
\pi_{\mstar}(E)\big(T,V\big)\xrightarrow{R_{f}} \pi_{\mstar}(E)\big(S,f^{\ast}V\big).
\]
\end{proposition}

The Weyl action here can be somewhat subtle. If \(V\) is a representation of \(G\), then we have a Weyl action on 
\[
\pi_{\mstar}(E)(G/H, G/H\times V).
\]
The standard isomorphism of \(G\)-spaces over \(G/H\)
\[
G/H\times V\cong G\timesover{H}i_{H}^{\ast}V
\]
give isomorphisms of vector bundles, and hence this group depends only on the \(H\)-representation \(i_{H}^{\ast}V\). The Weyl action, however, depends on \(V\) itself as a \(G\)-representation. Put another way, the standard isomorphism given above is not Weyl-equivariant.

\begin{example}
If \(G=C_{2}\) and \(V=(\sigma-\R)\), the virtual dimension zero shift of the sign representation, then the groups
\[
\pi_{\mstar}(E)\big(C_{2},C_{2}\times (V\oplus\mathbb R^{n})\big)
\]
are just the ordinary homotopy groups
\[
\pi_{n}(i_{e}^{\ast}E).
\]
In this case, however, we have twisted the Weyl action: as a \(C_{2}\)-module, we have an isomorphism
\[
\pi_{\mstar}(E)\big(C_{2},C_{2}\times (V\oplus\mathbb R^{n})\big)\cong \pi_{n}(i_{e}^{\ast}E)\otimes \sigma_{\mathbb Z},
\]
where \(\sigma_{\mathbb Z}\) is the integral sign representation. This observation has been used by many authors in the study of the homotopy fixed point spectral sequence for Hopkins--Miller spectra (see \cite{HMtmf13}).
\end{example}

\begin{remark}
The Mackey double coset formula also changes in the \(\RO\)-grading: there can be signs introduced which reflect the degree of the map on the underlying representation sphere. See, for example, \cite[Lemma 7.20]{HHR}.
\end{remark}

Smashing together maps gives us the external product.

\begin{definition}\label{def:ExternalProduct}
If \(x\in\pi_{\mstar}(E)(T,V)\) and \(y\in\pi_{\mstar}(E')(S,W)\), then we have an external product
\[
x\wedge y\in \pi_{\mstar}(E\wedge E')(T\times S,V\times W)
\]
given by the smash product of representing maps.
\end{definition}

Since this pairing is the one arising from the pairing of homotopy classes of functions in \(G\)-spectra, it has the usual properties.

\begin{proposition}
The external product is linear in both factors and satisfies the Frobenius relation:
\[
x\wedge T_{f}(y)=T_{Id\times f}(x\wedge y).
\]
\end{proposition}

The multiplication in the \(\RO\)-graded context can be a little more confusing, since elements are attached to virtual representations for different groups. To effectively compare them, the elements must first be restricted to a maximal common subgroup. In general, we have many ways to represent this. Conceptually, the \(\RO\)-graded group actually remembers more information, including not only the elements but also the various Weyl conjugates. Thinking in this way, the \(\RO\)-graded products will not only record the product we would expect but also include any of the pairwise products of restrictions to conjugate subgroups.

If \(T=S\), then we have a canonical pullback diagram
\[
\begin{tikzcd}
{V\oplus W}
	\ar[r]
	\ar[d]
	&
{V\times W}
	\ar[d]
	\\
{T}
	\ar[r, "\Delta_{T}"']
	&
{T\times T.}
\end{tikzcd}
\]
Composing the external product with the restriction along the diagonal \(\Delta_{T}\) gives the usual product structure on the \(\RO(T)\)-graded homotopy of the ``restriction to \(T\)'' of a ring spectrum \(R\).

If \(T=G/H\) and \(S=G/K\), with \(H\) and \(K\) not necessarily conjugate, then we do not have as simple a picture. The classes \(x\) and \(y\) are maps
\[
G_{+}\smashover{H}S^{V}\xrightarrow{x} E\quad\&\quad G_{+}\smashover{K}S^{W}\xrightarrow{y} E',
\]
and smashing them together gives the map
\[
\big(G_{+}\smashover{H}S^{V}\big)\wedge\big(G_{+}\smashover{K}S^{W}\big)\xrightarrow{x\wedge y} E\wedge E'.
\]
The source is naturally the Thom spectrum of a virtual bundle on \(G/H\times G/K\), which can be rewritten by the double coset formula as
\[
G/H\times G/K\cong G\timesover{H} i_{H}^{\ast}G/K\cong \coprod_{HgK\in H\backslash G/K} G/(H\cap gKg^{-1}).
\]
The bundle over the summand associated to \(HgK\) is 
\[
i_{H\cap gKg^{-1}}^{\ast} V\oplus i_{H\cap gKg^{-1}}^{\ast}c_{g}^{\ast}W,
\]
and the corresponding map on this summand is
\[
\res_{H\cap gKg^{-1}}^{H}(x)\wedge \res_{H\cap gKg^{-1}}^{gKg^{-1}}(c_{g}^{\ast}y).
\]

\begin{corollary}[{\cite{AngeltveitBohmann}}]
If \(E\) has a multiplication in the homotopy category of genuine \(G\)-spectra, then the composition with the multiplication in \(E\) makes \(\m\pi_{\mstar}(E)\) into an \(\RO\)-graded Green functor.
\end{corollary}

In fact, there is a good \(G\)-symmetric monoidal category of \(R_{\mstar}\)-modules for any equivariant commutative ring spectrum \(R\). This will be developed in forthcoming joint work with Angeltveit--Bohmann. We will make use of this structure somewhat heavily in what follows. However, the only cases in which we will consider it are ones for which the structure is immediate from the definition of the objects, so there should be no confusion.

We close by summarizing the notation that will show up for the various kinds of gradings.
\begin{notation}\mbox{}

The wildcard \(\mstar\) will be used for gradings by \({\RO}\).

The wildcard \(\star\) will be used for gradings by \(\RO(G)\).

The wildcard \(\ast\) will be used for gradings by \(\Z\).
\end{notation}

\section{Free \texorpdfstring{\(R\)}{R}-homology}

In this section, let \(R\) be a fixed \(E_{\infty}\)-ring spectrum in genuine \(G\)-spectra. Equivariantly, this is weaker than being a commutative monoid in any of the good point-set models of spectra, but this is sufficient to have a good, symmetric monoidal category of modules over \(R\) \cite{BHModules}. 

\subsection{Free and projective}
It greatly simplifies much of the notation (and of our discussion of a basis) to allow ourselves to evaluate our homotopy Mackey functors on infinite \(G\)-sets and virtual representations on these. 

\begin{notation}
If \(T\) is a discrete \(G\)-set and \(V\) is a virtual bundle over \(T\), then let
\[
{R}_{\mstar}(E)(T,V)=\lim_{\longleftarrow} {R}_{\mstar}(E)(S,i_{S}^{\ast}V),
\]
where \(S\) ranges over all finite subsets of \(T\).
\end{notation}

Since Thom spectra of disjoint unions of spaces is the coproduct of the associated Thom spectra, we have a natural isomorphism
\[
{R}_{\mstar}(E)(T,V)\cong \big[M(V),R\wedge E\big]^{G}.
\]

\begin{definition}
A \(G\)-spectrum \(E\) has {\defemph{free \(R\)-homology}} or ``is \(R\)-free'' if there is a \(G\)-set \(T_{E}\) and a virtual vector bundle \(V_{E}\) over \(T_{E}\)  such that we have an equivalence of \(R\)-modules
\[
R\wedge E \simeq R\wedge M(V_{E}).
\]
The full subcategory of \(Sp^{G}\) spanned by the spectra with free \(R\)-homology will be denoted
\[
\Sp^{G}_{R,fr}.
\]

It has {\defemph{projective \(R\)-homology}} if \(R\wedge E\) is a retract of an \(R\)-module of the form \(R\wedge M(V)\) for some virtual vector bundle \(V\) over a \(G\)-set. The full subcategory of \(Sp^{G}\) spanned by the spectra with projective \(R\)-homology will be denoted
\[
\Sp^{G}_{R,pr}.
\]
\end{definition}

\begin{remark}
The use of ``free'' here is to bring to mind a free module. In the homotopy category of \(R\)-modules, the \(R\)-module \(R\wedge (G_{+}\smashover{H}S^{V})\) corepresents the functor
\[
E\mapsto \pi_{V}^{H}(E),
\]
on the category of \(R\)-modules, and hence maps out of it correspond to certain elements in this \({\RO}\)-graded Mackey functor.
\end{remark}

\begin{definition}
If \(E\) has free \(R\)-homology, then a {\defemph{basis}} for the \(R\)-homology of \(E\) is an element
\[
\vec{x}\in R_{\mstar}(E)(T_{E},V_{E})
\] 
for a \(G\)-set \(T_{E}\) and a virtual bundle \(V_{E}\) over \(T_{E}\), such that the induced map
\[
R\wedge M(V_{E})\xrightarrow{R\wedge\vec{x}} R\wedge E
\]
is an equivalence.
\end{definition}

We can restate the definition of a basis using an orbit decomposition of \(T\). A choice of points in each orbit for \(T\) gives an equivariant isomorphism
\[
T\cong \coprod_{t\in T/G} G/H_{t},
\]
and if we let \(V_{E,t}\) be the restriction of \(V_{E}\) to the orbit \(G/H_{t}\), then
\[
{R}_{\mstar}(E)(T,V)\cong \prod_{t\in T/G}R_{V_{E,t}}^{H_{t}}(E).
\]
A basis then is a collection of elements 
\[
x_{t}\colon G_{+}\smashover{H_{t}}S^{V_{E,t}}\to R\wedge E
\]
such that the induced map
\[
R\wedge \left(\bigvee_{t\in T/G} G_{+}\smashover{H_{t}}S^{V_{E,t}}\right)\to R\wedge E
\]
is an equivalence. We will use both formulations.

Just as for vector spaces, a basis is a choice of additional data which aids in explicit computation. In particular, describing product structures is greatly simplified with a basis. 

\begin{remark}
Although a basis is given by specifying certain elements in certain \({\RO}\)-graded stems, the freeness actually gives us a host of related elements. Consider \(M=R\wedge (G_{+}\smashover{H} S^{V})\). The unit map \(S^{0}\to R\) gives a distinguished map
\[
G_{+}\smashover{H}S^{V}\to M
\]
which gives a basis. Unpacking the adjunction, this corresponds to the \(H\)-equivariant map
\[
S^{V}\to i_{H}^{\ast} R\wedge i_{H}^{\ast}(G)_{+}\smashover{H} S^{V}
\]
induced by the inclusion
\[
S^{V}\cong H_{+}\smashover{H} S^{V}\hookrightarrow i_{H}^{\ast}G_{+}\smashover{H} S^{V}.
\]
If \(H\neq G\), then there are many other summands. In particular, any element \(g\in N_{G}(H)\) gives us a summand 
\[
gH_{+}\smashover{H} S^{V},
\] 
which is the representation sphere for the pullback of the representation \(V\) along the conjugation-by-\(g\) automorphism of \(H\). When \(V\) is in the image of the restriction of a representation of \(N_{G}(H)\), this is just recording the Weyl-conjugates of our original element. 
\end{remark}

\begin{example}
For any \(R\), a basis for \(R\wedge (G_{+}\wedge S^{1})\) is the data of an element 
\[
x\in \pi_{1}\Big(\bigvee_{|G|}i_{e}^{\ast}\Sigma R\Big)\cong \bigoplus_{|G|}\pi_{0}\big(i_{e}^{\ast}R\big)
\] 
such that the induced action-map
\[
\Z[G]\otimes \pi_0(i_e^{\ast}R)\to \bigoplus_{|G|}\pi_{0}\big(i_{e}^{\ast}R\big)
\]
is an isomorphism. It also records (linearly independent) classes for all \(\gamma\in G\):
\[
\gamma\cdot x\in \bigoplus_{|G|}\pi_{0}(i_{e}^{\ast}R)
\] 
(and in fact, the Mackey functor homotopy group is \(\Ind_{e}^{G} \pi_{0}(i_{e}^{\ast}R)\)).
\end{example}

It is helpful to keep in mind the example of Bredon homology with coefficients in a commutative Green functor \(\mR\). The Eilenberg--Mac Lane spectrum associated to a commutative Green functor is always \(E_{\infty}\), so we can apply this general formalism.

\begin{example}\label{exam:Underlying}
If \(G=\{e\}\), then a basis for the homology of \(E\) with coefficients in \(\mR=R\) (an ordinary commutative ring) is the same as a basis for the graded \(R\)-module \(H_{\ast}(E;R)\).
\end{example}

\begin{example}
Kronholm and Hogle--May showed that if \(X\) is a finite \(Rep(C_2)\)-complex (meaning a \(C_2\)-complex formed by attaching disks in representations along their boundaries), then \(X\) has free \(H\m{\F}_2\)-homology (with no summands induced up from the trivial group) \cite{Kronholm}, \cite{HogleMay}.
\end{example}

\begin{example}
C.~May's decomposition theorem for the [co]homology of a finite \(C_{2}\)-CW complex says that for any finite \(C_{2}\)-complex \(X\), we have a splitting
\[
H\m\F_{2}\wedge X\simeq H\m\F_{2}\wedge\left(\bigvee C_{2+}\smashover{H}S^{V} \vee \bigvee \Sigma^{k_i}S(n_{i}\sigma)_{+}\right)
\]
where in the second sum, \(n_{i}\geq 2\) and \(\sigma\) is the sign representation \cite{MayDecomp}. Thus \(C_{2}\) spaces have free \(H\m\F_{2}\)-homology if and only if this second sum vanishes.
\end{example}

\begin{example}
Hazel's computation of the Bredon homology of \(C_{2}\)-surfaces shows that every connected \(C_{2}\)-surface for which the action is not free has free \(H\m\F_{2}\)-homology \cite[Theorem 6.6]{HazelSurfaces}. 
\end{example}

\begin{example}
Ricka extended the Hu--Kriz computation of the dual Steenrod algebra for \(\m\F_{2}\) and showed that \(H\m\F_{2}\) has free \(H\m\F_{2}\)-homology \cite{RickaSubalg}, \cite{HuKriz}.
\end{example}

There are two important restricted cases that show up often in computations. The proof of the following theorem is just by observation.
\begin{theorem}\label{thm:InducedGrading}
Let \(E\) be a spectrum with free \(R\)-homology.
\begin{enumerate}
\item If the basis comes from virtual representations of \(G\) exclusively, then the \({\RO}\)-graded homotopy is induced up from the \(\RO(G)\)-graded homotopy Mackey functors.
\item If the basis comes from trivial representations of \(G\), then the \({\RO}\)-graded homotopy is induced up from the \(\Z\)-graded homotopy Mackey functors.
\end{enumerate}
In both cases, we can extend to induced cells, provided we again only consider representations of \(G\) and trivial representations, respectively. 
\end{theorem}

\subsection{Closure properties of \texorpdfstring{$\Sp^{G}_{R,\fpr}$}{R free}}

The spectrum with \(R\)-free or projective homology enjoy a number of useful closure properties.

\begin{notation}
Let {\fpr} stand for either ``fr'' or ``pr''.
\end{notation}

\subsubsection{Closure under sums}

\begin{proposition}
The adjoint pair \(G_{+}\smashover{H}(\mhyphen)\dashv i_{H}^{\ast}\) on equivariant spectra descends to an adjoint pair
\[
G_{+}\smashover{H}(\mhyphen)\colon\Sp^{H}_{i_{H}^{\ast}R,\fpr}\rightleftarrows \Sp^{G}_{R,\fpr}\colon i_{H}^{\ast}.
\]
A basis for one gives the other via restriction or induction.
\end{proposition}
\begin{proof}
Since these are full subcategories, and since retracts are preserved by any functor, it suffices to show that restriction and induction preserve \(R\)-free spectra. For restriction, we just use the restriction of the Thom spectrum. For induction, we note the equivalence:
\[
G_{+}\smashover{H} M(V)\simeq M(G\timesover{H} V),
\]
where \(G\timesover{H}V\) is the induced bundle over \(G\timesover{H} T\).
\end{proof}

\begin{proposition}\label{prop:FreeForSums}
The category \(\Sp^{G}_{R,\fpr}\) is closed under arbitrary coproducts. 
A basis for the wedge is the sum of the bases.
\end{proposition}
\begin{proof}
The smash product distributes over wedges, and the wedge of Thom spectra of virtual bundles over \(G\)-sets is again a Thom spectrum of a virtual bundle over a \(G\)-set.
\end{proof}

\subsubsection{Closure under base-change}

The notions of free and projectives also work well with base-change.

\begin{proposition}\label{prop:FreeforBaseChange}
A map \(f\colon R\to R'\) of \(E_{\infty}\) ring spectra induces a map
\[
f_{\ast}\colon \Sp^{G}_{R,\fpr}\to \Sp^{G}_{R',\fpr}.
\]
A basis \(\vec{x}\) for \(E\) over \(R\) gives a basis \(f_{\ast}(\vec{x})\) be composing with \(f\).
\end{proposition}
\begin{proof}
This follows from base-changing the equivalence \(R\wedge E\simeq R\wedge M(V_{E})\) along the map \(R\to R'\).
\end{proof}

\subsubsection{Closure under products}
The categories of frees and projectives are also closed under the [twisted] smash products on \(G\)-spectra, being closed under the norms which \(R\) has.

\begin{proposition}\label{prop:FreeForProducts}
The category \(\Sp^{G}_{R,\fpr}\) is a symmetric monoidal subcategory of \(\Sp^{G}\) for the smash product.
A basis for \(E\) and \(E'\) gives a basis for \(E\wedge E'\) by boxing them together.
\end{proposition}
\begin{proof}
Again, it suffices to show this for free spectra. If \(V\) and \(V'\) are virtual vector bundles on \(T\) and \(T'\) respectively, then we have a natural equivalence
\[
M(V)\wedge M(V')\simeq M(V\times V'),
\]
where the latter is just the Thom spectrum of product of \(V\) and \(V'\) over \(T\times T'\). The result follows from recalling that the functor \(R\wedge(\mhyphen)\) is a strong symmetric monoidal functor from \(G\)-spectra to \(R\)-modules.
\end{proof}

This gives us a kind of weak K\"unneth theorem.
\begin{theorem}\label{thm:FreeForModules}
If \(E\in\Sp^{G}_{R,\fpr}\), then for any \(R\)-module \(M\), then \(M\wedge E\) is a summand of \(M\wedge M(V_{E})\) for some virtual vector bundle \(V_{E}\), and hence the  multiplication gives a natural isomorphism
\[
R_{\mstar}(E)\boxover{R_{\mstar}} {\pi}_{\mstar}(M)\to {\pi}_{\mstar}(E\wedge M).
\]
\end{theorem}
\begin{proof}
Again, it suffice to show for \(E\) \(R\)-free. By assumption, there is a splitting of \(R\)-modules
\[
R\wedge M(V_{E})\simeq R\wedge E.
\]
This gives an equivalence of \(R\)-modules
\[
E\wedge M\simeq (R\wedge E)\smashover{R} M\simeq \big(R\wedge M(V_{E})\big)\smashover{R} M
\simeq M(V_{E})\wedge M.
\]
Since the smash product distributes over the wedge, the latter spectrum is a wedge of \(R\)-modules of the form
\[
(G_{+}\smashover{H}S^{V})\wedge M.
\]
The result follows by the definition of the representables.
\end{proof}

If the basis is via representations of \(G\) or trivial representations, then this recovers the K\"unneth theorem of Lewis--Mandell \cite{LewisMandell}.

\begin{corollary}
Let \(E\) be in \(\Sp^{G}_{R,\fpr}\) and let \(M\) be an \(R\)-module. 
\begin{enumerate}
\item If a basis for \(E\) can be chosen such that only virtual representations of \(G\) are used, then we have an isomorphism
\[
\m{\pi}_{\star}(E\wedge M)\cong \m{R}_{\star}(E)\boxover{\m{R}_{\star}} \m{\pi}_{\star}(M).
\]
\item If a basis for \(E\) can be chosen such that only trivial representation of \(G\) are used, then we have an isomorphism
\[
\m{\pi}_{\ast}(E\wedge M)\cong \m{R}_{\ast}(E)\boxover{\m{R}_{\ast}} \m{\pi}_{\ast}(M).
\]
\end{enumerate}
\end{corollary}

In the special case that the module \(M\) is in fact \(R\wedge E'\) for some \(R\)-free or projective \(E'\), this shows that the functor of \({\RO}\)-graded \(R\)-homology Mackey functors is strong symmetric monoidal.

\begin{corollary}
If \(E, E'\in\Sp^{G}_{R,\fpr}\), then the multiplication gives a natural isomorphism
\[
R_{\mstar}(E\wedge E')\cong R_{\mstar}(E)\boxover{R_{\mstar}} R_{\mstar}(E').
\]
\end{corollary}

%

\subsubsection{Closure under norms}

For the norms, we recall some properties of the norm and these relatively simple Thom spectra.

\begin{notation}
If \(T\) is an \(H\)-set and \(V\to T\) is a virtual vector bundle, then let
\[
\Map^{H}\!(G,V)\to\Map^{H}\!(G,T)
\]
be the coinduced vector bundle over \(\Map^{H}\!(G,T)\).
\end{notation}

\begin{proposition}\label{prop:NormsofThom}
For any virtual vector bundle \(V\), we have 
\[
M\big(\Map^{H}\!(G,V)\big)\simeq N_{H}^{G}M(V).
\]
\end{proposition}
\begin{proof}
All of the functors considered commute with filtered colimits, so it suffices to consider the case that \(T\) is finite. This is then the distributive law for norms, together with the observation that if \(T=H/K\), then \(M(V)=H_+\smashover{K}S^V\), allows us to further reduce to the case that \(T=H/H\). A vector bundle over this is just a virtual \(H\)-representation, the Thom spectrum of which is the corresponding representation sphere. The coinduced space in this case is \(G/G\), and the representation is \(\ind_H^G V\). We now compute
\[
M\big(\Map^{H}\!(G,V)\big)=S^{\ind_{H}^{G}V}\simeq N_{H}^{G}S^{V}\cong N_{H}^{G}M(V).\qedhere
\]
\end{proof}

The norm is also a strong symmetric monoidal functor, and hence it induces a map
\[
N_{H}^{G}\colon \RMod\to N_{H}^{G}\RMod.
\]
\begin{proposition}
The norm induces a functor
\[
N_{H}^{G}\colon \Sp^{H}_{R,\fpr}\to \Sp^{G}_{N_{H}^{G}R,\fpr}.
\]
\end{proposition}
\begin{proof}
It suffices to show this for \(E\) having free \(R\)-homology. In this case, we simply apply the norm to the equivalence
\[
R\wedge E\simeq R\wedge M(V_{E})
\]
for some virtual vector bundle \(V_{E}\) and use Proposition~\ref{prop:FreeForProducts}.
\end{proof}

If \(R\) is an \(E_{\infty}\) ring spectrum that has an \(E_{\infty}\)-map
\[
N_{H}^{G}i_{H}^{\ast}R\to R,
\]
then we have a relative norm map on \(R\)-modules given by
\[
M\mapsto R\smashover{N_{H}^{G}i_{H}^{\ast}R} N_{H}^{G}M.
\]
The usual case is when \(R\) is an equivariant commutative ring spectrum (i.e. a \(G\)-\(E_{\infty}\) ring spectrum), but this has also been worked out for algebras over linear isometries operads \cite{BHModules}.

\begin{proposition}
Let \(R\) be an \(E_{\infty}\) ring spectrum that has an \(E_{\infty}\)-map
\[
N_{H}^{G}i_{H}^{\ast}R\to R,
\]
then \(N_{H}^{G}\) induces a functor
\[
\Sp^{H}_{i_{H}^{\ast}R,\fpr}\to\Sp^{G}_{R,\fpr}.
\]
The norm of a basis for \(E\) gives one for the norm.
\end{proposition}

\begin{example}\label{exam:NormsOfMU}
Since 
\[
\MU\wedge\MU\simeq\MU\wedge BU_{+}\simeq\MU[b_{1},\dots],
\]
where \(|b_{i}|=2i\), the spectrum \(\MU\) and the space \(BU\) have free \(\MU\)-homology. This implies that the same is true for the norms:
\(
N_{e}^{G}\MU
\) and \(N_{e}^{G}\Sigma^{\infty}_{+}BU\) have free \(N_{e}^{G}\MU\)-homology.

We have identical statements for \(\mathbb CP^{n}\) for all \(n\leq \infty\) and the spaces \(BU(n)\).
\end{example}

Using the orientations given by the norm of \(\MU\), we produce a host of other interesting examples.

\begin{example}
Let \(R\) be an \(E_{\infty}\) \(G\)-spectrum that admits an \(E_{\infty}\) norm map
\[
N_{e}^{G}i_{e}^{\ast}R\to R,
\]
and assume that \(i_{e}^{\ast}R\) can be given a commutative complex orientation. Then for any spectrum \(E\) such that \(\MU_{\ast}E\) is a free \(\MU_{\ast}\)-module, \(N_{e}^{G}(E)\) has free \(R\)-homology.

The identity map \(\MU\to i_{e}^{\ast}\MU_{G}\) and the Connor--Floyd map \(\MU\to KU\) give examples of commutative complex orientations, which shows that the spaces and spectra considered in Example~\ref{exam:NormsOfMU} have free \(\MU_{G}\) and \(KU_{G}\)-homology.
\end{example}

\begin{example}
If \(E\) is any finite type, bounded below spectrum with free integral homology, then \(N_{e}^{G}E\) has free \(KU_{G}\) and \(\MU_{G}\) homology.
\end{example}

In the Bredon case, if \(\mR\) has the structure of a Tambara functor \cite{Tambara}, then Ullman has shown that \(H\mR\) has the structure of a \(G\)-\(E_{\infty}\) ring spectrum \cite{Ullman}. This gives us many examples for Bredon homology. In particular, the absolute norms (i.e. the norms from the trivial group) of an ordinary commutative ring are always Tambara functors. Generalizing the \(C_{2}\)-equivariant examples of \cite{SignedLoops}, we get that absolute norms are free for a large number of Green functors.

\begin{example}
Let \(k\) be a field and let \(\mR\) be a Green functor under \(N_{e}^{G}k\). Then for any spectrum \(E\), \(N_{e}^{G}E\) has free \(H\mR\)-homology.
\end{example}

The more general integral story also follows.

\begin{example}
If \(E\) is an ordinary, non-equivariant spectrum such that \(H_{\ast}(E;\mathbb Z)\) is free, then \(N_{e}^{G}E\) has free \(\mA\)-homology. Since \(H\mM\) is an \(H\mA\)-module for any Mackey functor \(\mM\), \(N_{e}^{G}E\) has free \(\mM\)-homology for any \(\mM\).
\end{example}

There is a norm in \(R\)-homology, specified by the norms in Mackey functors (or equivalently in spectra), and the following holds by definition.

\begin{corollary}
If \(E\) is an \(H\)-spectrum that has free \(i_{H}^{\ast}R\)-homology, then we have a natural isomorphism
\[
R_{\mstar}(N_{H}^{G}E)\cong N_{H}^{G}\big(i_{H}^{\ast}R_{\mstar}(E)\big).
\]
\end{corollary}

\begin{notation}
In any context where it is defined, let \(N^{G/H}\) be the composite \(N_{H}^{G}i_{H}^{\ast}\).
\end{notation}

With this notation, if \(E\) is a \(G\)-spectrum with free \(R\)-homology, then we have a natural isomorphism
\[
R_{\mstar}(N^{G/H}E)\cong N^{G/H} R_{\mstar}E.
\]

As a specific example, this gives us the equivariant homology of the topological Singer construction \cite{LunRog}.

\begin{example}\label{exam:TopSinger}
Let \(k=\mathbb F_{p}\), let \(G=C_{p}\), and let \(\m\F_{p}\) be the constant Green functor \(\F_{p}\). This is a Tambara functor, so for any spectrum \(E\), we have a natural isomorphism
\[
H\big(N_{e}^{C_{p}}(E);\m\F_{p}\big)\cong N_{e}^{C_{p}}\big(H_{\ast}(E;\mathbb F_{p})\big). 
\]
In particular, for \(p=2\), the \(\m{\F}_{2}\)-Bredon homology of \(N_{e}^{C_{2}}H\F_{2}\) is free.

Unpacking this a little more, a basis is given by orbits \([f]\) in the monomial basis in
\[
\Big(\mathbb F_{p}[\xi_{1},\dots]\otimes E(\tau_{0},\dots)\Big)^{\otimes p}\!\Big/C_{p},
\]
where the group \(C_{p}\) acts on this by permuting the tensor factors. Every monic monomial \(f\) has a stabilizer subgroup \(H_{f}\). This is the subgroup associated to the orbit \([f]\) as a basis vector. The degree of \(f\) is given by 
\[
||f||=\frac{|f|}{|H_{f}|}\rho_{H_{f}},
\]
where \(|f|\) is the ordinary, underlying degree induced by the degrees in the dual Steenrod algebra.
\end{example}

These freeness results can also give us interesting information about non-free spectra. Snaith showed that we have an equivalence of \(E_{\infty}\) ring spectra
\[
KU\simeq \Sigma^{\infty}_{+}\mathbb CP^{\infty}[\beta^{-1}],
\]
where \(\beta\) is the map on \(\Sigma^{\infty}_{+}\mathbb CP^{\infty}\) induced by the inclusion \(\mathbb CP^{1}\hookrightarrow \mathbb CP^{\infty}\) \cite{SnaithKU}. 

The norm functor commutes with filtered colimits, so this gives us an equivariant version of Snaith's theorem.

\begin{theorem}
For any finite group \(G\), we have an equivalence of \(G\)-\(E_{\infty}\) ring spectra
\[
N_{e}^{G}KU\simeq\Sigma^{\infty}_{+}\Map(G,\mathbb CP^{\infty})[N(\beta)^{-1}],
\]
where 
\[
N(\beta)\colon S^{2\rho_{G}}\to \Sigma^{\infty}_{+}\Map(G,\mathbb CP^{\infty})
\]
is induced by the norm. 
\end{theorem}

\begin{corollary}
Let \(R\) be a \(G\)-\(E_{\infty}\) ring spectrum such that \(\mathbb CP^{\infty}\) has free \(i_{e}^{\ast}R\)-homology. Then we have an isomorphism
\[
R_{\mstar} N_{e}^{G}KU\simeq \lim_{\longrightarrow} \Sigma^{-2n\rho_{G}} N_{e}^{G}\big(i_{e}^{\ast}R_{\ast}(\mathbb CP^{\infty})\big).
\]
In particular, this is always a flat \(R_{\mstar}\)-module.
\end{corollary}
\begin{proof}
Both the norm and \(R\)-homology commute with direct limits. This follows from the homology of the {\coinduction}.
\end{proof}

\begin{example}
%
Because \(H\mA\) is a \(G\)-\(E_{\infty}\)-commutative ring spectrum and \(\mathbb CP^{\infty}\) has free \(H\Z\)-homology, we have
\[
H_{\mstar}\big(\Map(G,\mathbb CP^{\infty});\mA\big)\cong N_{e}^{G} H_{\ast}(\mathbb CP^{\infty};\mathbb Z)\cong N_{e}^{G}\big(\Gamma(x)\big),
\]
where \(|x|=2\). The Bott element we invert is the norm of \(x\), and we deduce 
\[
H_{\mstar} N_{e}^{G}KU\simeq N_{e}^{G} \mathbb Q[x^{\pm 1}].
\]

Since \(KU_{G}\) is a \(N_{e}^{G}KU\)-algebra, we also deduce that \(H\mA\wedge KU_{G}\) is rational.
\end{example}

\begin{remark}
The norms of the divided power algebra are curious Tambara functors, though the structure can be worked out from the basic properties of the norms. We spell this out for \(C_{p}\). 

Just as with the discussion of the topological Singer construction in Example~\ref{exam:TopSinger}, a basis is given by orbits of monic monomials in the tensor power \(\Gamma(x)^{\otimes p}\), where \(C_{p}\) again acts by permutation. The stabilizer of a monic monomial again gives us the appropriate way to determine the degrees: most monomials are stabilized by \(\{e\}\) and hence correspond to a free summand, while monomials of the form \(f^{\otimes p}\) correspond to a summand \(S^{|f|\rho_{p}}\). These are the norms of classes \(f\otimes 1^{\otimes (p-1)}\).

Let \(\gamma_{i}(x)\) be the \(i\)th divided power of \(x\). Then we have the divided power relations
\[
\gamma_{i}x\cdot\gamma_{j}x=\binom{i+j}{i}\gamma_{i+j}x.
\]
Since the norm is multiplicative, this gives relations
\[
N(\gamma_{i}x)\cdot N(\gamma_{j}x)=N\binom{i+j}{i}N(\gamma_{i+j}x).
\]
The norms of integers can be computed back in the Burnside Mackey functor, where we find
\[
N_{e}^{C_{p}}(n)=n+\frac{n^{p}-n}{p}[C_{p}],
\]
and this gives us the actual relation:
\[
N(\gamma_{i}x)\cdot N(\gamma_{j}x)=\binom{i+j}{i}N(\gamma_{i+j}x)+\frac{\left(\binom{i+j}{i}^{p}-\binom{i+j}{i}\right)}{p} \tr_{e}^{C_{p}}\big(\gamma_{i+j}(x)^{\otimes p}\big).
\]
\end{remark}

\subsubsection{Closure under duals}
We also have a weak Universal Coefficients Theorem, provided our spectrum is small.

\begin{definition}
Let \(E\in\Sp^{G}_{R,\fpr}\), and let \(V_{E}\) be the associated virtual bundle such that \(R\wedge E\simeq R\wedge M(V_{E})\). We say \(E\) is {\defemph{finite type}} if for each \(k\leq j\in\mathbb Z\), only finitely many orbits of \(T_{E}\) contribute to \(\m\pi_{\ell}(R\wedge M(V_{E}))\) for \(k\leq \ell\leq j\). 
\end{definition}

Clearly, if the set \(T_{E}\) can be chosen to be finite, then it is finite type. This more general condition is analogous to only having finitely many cells in each degree.

\begin{theorem}
If \(E\in\Sp^{G}_{R,\fpr}\) is a finite complex, then \(D(X)\) is also in \(\Sp^{G}_{R,\fpr}\).

More generally, if \(E\in\Sp^{G}_{R,\fpr}\) is finite type, then for any \(R\)-module \(M\), we have a weak equivalence of \(R\)-modules
\[
F(E,M)\simeq M\wedge M(-V_{E}).
\]

We have a universal coefficients isomorphism that computes the \(M\)-cohomology of \(E\) out of the \(R\)-homology of \(E\):
\[
M^{-\m{\star}}(E)\cong \Hom_{\m{R}_{\m{\star}}}\!\!\big(R_{\m{\star}}(E),M_{\m{\star}}\big).
\] 
\end{theorem}
\begin{proof}
If \(E\) is a finite complex, then
\[
D(E)\wedge R\simeq F(E,R),
\]
and the first will follow from the second.

Since \(M\) is an \(R\)-module, we have an equivalence
\[
F(E,M)\simeq F_{R}(R\wedge E,M).
\]
A basis for the \(R\)-homology of \(E\) gives an equivalence 
\[
F_{R}(R\wedge E,M)\simeq F_{R}\big(R\wedge M(V_{E}),M\big),
\]
and this is equivalent to \(F\big(M(V_{E}),M\big)\). Since maps out of a wedge is the product, we first check the case of an orbit. The result is then the classical Wirthm\"uller isomorphism:
\[
F\big(G_{+}\smashover{H}S^{V},M\big)\simeq G_{+}\smashover{H}S^{-V}\wedge M.
\]
Finally, the finite type condition ensures that the natural map from the wedge to the product is in fact an equivalence.

The second part follows from this by taking homotopy and observing the result for orbits.
\end{proof}

A surprising final feature of the universal coefficients theorem is that we can also describe the cohomology of the norms of \(R\)-free spectra.

\begin{proposition}
Let \(R\) be an \(E_{\infty}\) ring spectrum that has an \(E_{\infty}\)-map
\[
N_{H}^{G}i_{H}^{\ast}R\to R.
\]
If an \(H\)-spectrum \(E\) has free \(i_{H}^{\ast}R\)-homology with a finite basis, then the function spectrum \(F\big(N_{H}^{G}E,R\big)\) is equivalent to a free \(R\)-module, and the basis is the dual to the one for \(N_{H}^{G}E\).
\end{proposition}

In particular, analyzing the Thom spectrum for the functional dual, we have that for \(E\) as in the proposition, the \(R\)-cohomology of \(N_{H}^{G}E\) can be described as the norm of the \(i_{H}^{\ast}R\)-cohomology of \(E\).

\subsubsection{Pullbacks}
Finally, freeness and projectivity is also preserved by restricting along quotient maps (also called ``pulling back''). 
\begin{notation}
If \(N\) is a normal subgroup of \(G\) and \(q\colon G\to Q=G/N\), then let
\[
q^{\ast}\colon\Sp^{Q}\to \Sp^{G}
\]
be the inclusion of \(Q\)-spectra into \(G\)-spectra.
\end{notation}

\begin{proposition}\label{prop:Pullbacks}
The functor \(q^{\ast}\) induces a functor
\[
q^{\ast}\colon\Sp^{Q}_{R,\fpr}\to \Sp^{G}_{q^{\ast}R,\fpr}.
\]

If \(E\in\Sp^{Q}\) has a basis for \(R\), then \(q^{\ast}E\) has a basis for \(q^{\ast}R\).
\end{proposition}
\begin{proof}
Again, it suffices to check on the full subcategory of \(R\)-free spectra, and since \(q^{\ast}\) is strong symmetric monoidal, it suffices to show on the associated Thom spectra. By construction, 
\[
q^{\ast}M(V_{E})\simeq M(q^{\ast}V_{E}),
\] 
where \(q^{\ast}V_{E}\) is just \(V_{E}\) viewed as a \(G\)-virtual bundle.
\end{proof}

\begin{remark}
The fixed points functors do not preserve projective objects, as the tom Dieck splitting shows. However, the canonical map
\[
q^{\ast}(R^{G})\to R
\]
gives us a map
\[
q^{\ast}\colon\Sp_{R^{G},\fpr}\to\Sp^{G}_{R,\fpr}.
\]
\end{remark}

This gives another proof of a basic construction in Bredon homology.
\begin{example}
If \(E\) is an ordinary, non-equivariant spectrum such that \(H_{\ast}(E;\mathbb Z)\) is free in each degree, then \(q^{\ast}E\) has free Bredon homology for any coefficients:
\[
H_{\ast}(q^{\ast}E;\mM)\cong H_{\ast}(E;\Z)\otimes \mM.
\] 
For any \(G\), 
\[
{\pi}_{0}q^{\ast}H\mathbb Z=\m{A},
\]
and the negative homotopy groups are all zero. The zeroth Postnikov section then gives us an \(E_{\infty}\)-map
\[
q^{\ast}H\mathbb Z\to H\mA.
\]
The result then follows from Proposition~\ref{prop:Pullbacks}, Proposition~\ref{prop:FreeforBaseChange}, and Theorem~\ref{thm:FreeForModules}.
\end{example}

\begin{example}
If \(E\) is an ordinary, non-equivariant spectrum such that \(H_{\ast}(E; \mathbb F_{p})\) is free in each degree, then for any \(G\) and for any Green functor \(\mR\) in which \(p\cdot 1=0\in\mR(G/G)\), \(q^{\ast}E\) has free \(\mR\)-homology. This is because the pullback of \(H\mathbb F_{p}\) has \({\pi}_{0}=\mA/p\), the initial example of such a Green functor. 

In particular, for any \(G\) and for any \(\mR\) of this form, this applies to 
\[
E=\Sigma^{\infty} K(\mathbb F_{p},m),
\] 
the pullback of which is the suspension spectrum of the {\EM} space for the constant coefficient system \(\m{\F}_{p}\).
\end{example}

Both of these examples are also free with bases in integer stems. In particular this functorially describes the \({\RO}\)-graded homology, by Theorem~\ref{thm:InducedGrading}.

\subsection{Freeness and spaces}
The author's primary interest in these freeness results comes from the connection between the [twisted] smash products in spectra and [twisted] Cartesian products in spaces. 

\begin{proposition}\label{prop:SuspensionSpectrumGSym}
If \(X\) is an \(H\)-space, then we have a natural equivalence
\[
N_{H}^{G}\Sigma^{\infty}_{+}X\simeq\Sigma^{\infty}_{+}\Map^{H}\!(G,X).
\]

If \(X\) and \(Y\) are \(G\)-spaces, then 
\[
\Sigma^{\infty}_{+}(X\times Y)\simeq \Sigma^{\infty}_{+}X\wedge\Sigma^{\infty}_{+}Y.
\]
\end{proposition}

We can assemble all of our results so far into a summary theorem.

\begin{theorem}\label{thm:SymMonoidalFreeForSpaces}
Let \(X\) be an \(K\)-space such that \(X\) has free \(i_{K}^{\ast}R\)-homology. Then we have a natural isomorphism
\[
N_{K}^{G}R_{\m{\star}}(X)\cong R_{\m{\star}}\big(\Map^{H}\!(G,X)\big),
\]
and moreover, this is free on the basis \(N_{H}^{G}\vec{x}\), where \(\vec{x}\) is a basis for the homology of \(X\).

If \(X\) and \(Y\) are \(G\)-spaces that have \(R\)-free homology, then
\[
R_{\m{\star}}(X\times Y)\cong R_{\m{\star}}(X)\boxover{R_{\m\star}} R_{\m{\star}}(Y),
\]
with a basis given by the product of the bases.
\end{theorem}

\begin{example}
In general, {\coinduction} preserves {\EM} spaces: if \(\mM\) is an \(H\)-Mackey functor, then we have an equivalence
\[
\Map^{H}\!\big(G,K(\mM,n)\big)\simeq K\big(\Ind_{H}^{G}\mM,n\big).
\]
(More generally, the \(G\)-space \(\Map^{H}\!\big(G,K(\mM,V)\big)\) represents the functor 
\[
X\mapsto H^{V}(i_{H}^{\ast}X;\mM),
\]
so these are all kinds of {\EM} spaces.)

When \(H=\{e\}\), this allows us to determine the homology of {\EM} space attached to any induced Mackey functor with coefficients in a \(N_{e}^{G}k\)-algebra, for \(k\) a field. As an applicatioin, we have
\[
H_{\mstar}\Big(K\big(\Ind_{e}^{C_{p}}\F_{p},n\big);\m{\F}_{p}\Big)\cong N_{e}^{C_{p}}\Big(H_{\ast}\big(K(\F_{p},n);\F_{p}\big)\Big),
\]
and the latter was determined by Cartan and Serre \cite{CartanEM2}, \cite{SerreEM}.
\end{example}

This is closely connected to some additional structure that is often difficult to access. Equivariant spaces are canonically \(G\)-cocommutative comonoids. In addition to the coproduct
\[
X\to X\times X,
\]
they have conorm maps
\[
X\xrightarrow{\Delta^{G/H}} \Map(G/H,X)\cong\Map^{H}\!(G,X). 
\]
The contravariant Yoneda functor gives for any \(X\) a functor
\[
\yo_{X}\colon \Map(\mhyphen,X)\colon (\Fin^{G})^{op}\to\Top^{G},
\]
and on passage to fixed points, these conorm maps are exactly giving the usual coefficient system of fixed points for any \(G\)-space. 
\begin{definition}
If \(f\colon S\to T\) is a map of finite \(G\)-sets, then let
\[
\psi_{f}\colon R_{\m{\star}}\big(\Map(T,X)\big)\to R_{\m{\star}}\big(\Map(S,X)\big)
\]
be the ``conorm'' map associated to \(f\). When \(f=\nabla_{S}\colon S\amalg S\to S\) is the fold map, we call this the ``coproduct''.
\end{definition}
In general, this is difficult to work with, since we need not have a good [twisted] K\"unneth theorem. In the case we are considering, however, we do!

\begin{theorem}
Let \(R\) be an equivariant commutative ring spectrum, and let \(X\) be a space that has free \(R\)-homology. Then \(R_{\m{\star}}(X)\) has a comultiplication map
\[
R_{\m{\star}}(X)\to R_{\m{\star}}(X)\boxover{R_{\m{\star}}} R_{\m{\star}}(X)
\]
making it a ``co-Green functor''. Moreover, we have for any map of finite \(G\)-sets \(f\colon S\to T\) a conorm map 
\[
N^{T} R_{\m{\star}}(X)\to N^{S}R_{\m{\star}}(X)
\]
which is a map of co-Green functors.
\end{theorem}
\begin{proof}
Since \(X\) has \(R\)-free homology, so do all of its restrictions, and hence so do all of the spaces \(\Map(T,X)\) for any finite \(G\)-set \(T\). The comultiplication and conorm maps then follow immediately from our earlier analysis of the homology of the spaces involved.

That the conorm maps are maps of coGreen functors follows from naturality: for any \(f\colon S\to T\), we have a commutative diagram
\[
\begin{tikzcd}
{S\amalg S}
	\ar[r, "\nabla"]
	\ar[d, "f\amalg f"']
	&
{S}
	\ar[d, "q"']
	\\
{T\amalg T}
	\ar[r, "\nabla"']
	&
{T.}
\end{tikzcd}
\]
Functoriality then shows that the conorm associated to \(f\) is a map of co-Green functors.
\end{proof}

\begin{remark}
If a basis for \(R_{\mstar}(X)\) is in integer stems, then the co-Green structure is simply base-changed from the integral one. The conorms essentially never are, due to the degree scaling aspects of the norm.
\end{remark}

Rephrased, a space with free \(R\)-homology gives a strong \(G\)-symmetric monoidal functor
\[
\Set^{G,op}\to R_{\mstar}\mhyphen\Mod,
\]
where the \(G\)-symmetric monoidal structure on \(\Set^{G,op}\) is the dual to the co-Cartesian one. This is the definition of a \(G\) co-commutative comonoid. Since a Tambara functor is a \(G\)-commutative monoid by work of Mazur and Hoyer \cite{HMazur, HoyerThesis}, this gives \(R_{\m{\star}}(X)\) naturally the structure of a co-Tambara functor. 
Via the Universal Coefficients theorem for \(R\)-free spaces, this structure is dual to the structure which gives rise to the Tambara functor structure on the \(R\)-cohomology of a \(G\)-space.

\begin{remark}
The usual formulation of a Tambara functor describes norm maps \(n_{H}^{K}\colon\mR(G/H)\to\mR(G/K)\) with satisfy certain axioms relating them to the additive Mackey functor structure. These connect, via work of Mazur and Hoyer (\cite{HMazur}, \cite{HoyerThesis}) to \(G\)-commutative monoids in Mackey functors via canonical maps of sets
\[
\mR(G/H)\cong i_{K}^{\ast}\mR(K/H)\to (N^{K/H}i_{K}^{\ast}\mR)(K/K).
\]
The maps go the wrong way to be able to interpret a co-Tambara functor easily in the more traditional way: there is no clear way to extract a map
\[
\mR(G/K)\to \mR(G/H)
\]
from the data:
\[
\begin{tikzcd}
{i_{K}^{\ast}\mR(K/K)}
	\ar[r, "{\psi_{K}^{H}}"]
	&
{\big(N_{H}^{K}i_{K}^{\ast}\mR\big)(K/K)}
	\\
{}
	&
{i_{K}^{\ast}\mR(K/H).}
	\ar[u]
\end{tikzcd}
\]
\end{remark}

\subsection{Hopf algebroids and comodule Tambara functors}
There are many examples of spectra for which we have various kinds of comultiplications. These are ubiquitous amongst spectra which have free homology over themselves. This was shown and used by Hu--Kriz for a stricter notion of free; it works very generally.

\begin{proposition}[{\cite{HuKriz}}]
If \(R\) is an \(E_{\infty}\) ring spectrum such that has free \(R\)-homology, then the pair 
\[
(R_{\mstar},R_{\mstar}R)
\]
forms a Hopf algebroid, and moreover, the \(R\)-homology of any space or spectrum is a comodule over this.
\end{proposition}
\begin{proof}
Since \(R\) has free \(R\)-homology, we can apply the weak K\"unneth theorem to deduce a natural isomorphism for any \(E\)
\[
\pi_{\mstar}(R\wedge R\wedge E)\cong (R_{\mstar}R)\boxover{R_{\mstar}} R_{\star}(E).
\]
Applying this to the case \(E=R\) and considering the unit map in the middle copy of \(R\), we have the comultiplication. Applying this to a general \(E\) and again considering the unit on the rightmost copy of \(R\) give the coaction.
\end{proof}

It is important to note here that there are no hypotheses places on \(E\): the \(R\)-homology of any \(E\) inherits this structure. When \(E\) itself has more structure, then we can say even more about the coaction map. Classically, if \(E\) is a ring object in the homotopy category, then the \(R\)-homology of \(E\) is a comodule algebra, since the unit map is a map of ring spectra. When \(E\) is a \(G\)-commutative monoid in the homotopy category, then we have the analogous Tambara case.

\begin{theorem}
Let \(R\) be a \(G\)-\(E_{\infty}\) ring spectrum which has free \(R\)-homology and let \(E\) be a \(G\)-commutative monoid in the homotopy category. Then \(R_{\mstar}E\) is an \({\RO}\)-graded Tambara functor and the coaction map
\[
R_{\mstar}E\to R_{\mstar}R\boxover{R_{\mstar}} R_{\mstar}E
\]
is a map of \({\RO}\)-graded Tambara functors.
\end{theorem}
\begin{proof}
By assumption, \(R\wedge E\) is a \(G\)-commutative monoid in the homotopy category, and so is \(R\wedge R\wedge E\). The coaction map is the map induced by the unit in the middle, and this is a map of \(G\)-commutative monoids:
\[
R\wedge E\cong R\wedge S^{0}\wedge E\to R\wedge R\wedge E.
\]
The result follows from \cite{AngeltveitBohmann}.
\end{proof}

When \(E\) also has free \(R\)-homology, we can use an external version of the norm, making some of the structure more transparent. To avoid clutter, we restrict exposition to the norms from subgroups to \(G\). The more general ones follow from considering instead various restrictions to subgroups.

\begin{proposition}
Let \(R\) be a \(G\)-\(E_{\infty}\) ring spectrum which has free \(R\)-homology and let \(E\) be a \(G\)-commutative monoid in the homotopy category that has free \(R\)-homology. Then for any subgroup \(K\), we have a commutative diagram of commutative Green functors
\[
\begin{tikzcd}
{N^{G/K} R_{\mstar}(E)}
	\ar[r, "N^{G/K}\psi"]
	\ar[d, "N"']
	&
{N^{G/K}(R_{\mstar}R)\boxover{R_{\mstar}} N^{G/K}(R_{\mstar}E)}
	\ar[d, "N\Box N"]
	\\
{R_{\mstar}E}
	\ar[r, "\psi"']
	&
{R_{\mstar}R\boxover{R_{\mstar}} (R_{\mstar}E).}
\end{tikzcd}
\]
\end{proposition}

As an application of this structure, we can look at the coaction on the homology of the topological Singer construction at the prime \(2\).

\begin{example}
If \(R\) is a \(C_{2}\)-\(E_{\infty}\) ring spectrum that has free \(\m\F_{2}\)-homology, then \(H_{\mstar}(R;\m\F_{2})\) is a comodule Tambara functor over the equivariant dual Steenrod algebra: the comodule structure map is a map of Green functors and commutes with the norms:
\[
\begin{tikzcd}
{N_{e}^{C_{2}}H_{\ast}(i_{e}^{\ast}R;\F_{2})}
	\ar[r, "N_{e}^{C_{2}}\psi"]
	\ar[d, "N"']
	&
{\big(N_{e}^{C_{2}}\cA_{\ast}\big)\boxover{H_{\mstar}} N_{e}^{C_{2}} H_{\ast}(i_{e}^{\ast}R;\F_{2})}
	\ar[d, "N\Box N"]
	\\
{H_{\mstar}(R;\m\F_{2})}
	\ar[r, "\psi"']
	&
{\cA_{\mstar}\Boxover{H_{\mstar}} H_{\mstar}(R;\m\F_{2})}.
\end{tikzcd}
\]
This means in particular that the coaction on the spectrum \(N_{e}^{C_{2}}H\F_{2}\) is completely determined by the coaction on \(H\F_{2}\), allowing us to analyze the homotopy groups of this spectrum by a Hu--Kriz style Adams spectral sequence \cite{HuKriz}.
\end{example}

\section{An even nicer class of spectra}
\subsection{Homological purity}
We single out a class of spectra for which computations are strikingly simple, being completely determined by the homology of the underlying spectrum.
\begin{definition} 
A {\defemph{regular slice sphere}} is a \(G\)-spectrum of the form
\[
G_+\smashover{H}S^{k\rho_H},
\]
for some integer \(k\). The dimension of such a regular slice sphere is \(k|H|\).
\end{definition}

In \cite{HHR}, a spectrum \(E\) was called ``pure'' if the slice associated graded of \(E\) is a wedge of regular slice spheres smashed with \(H\mZ\). We build on that here.

\begin{definition}
A \(G\)-spectrum \(E\) is {\defemph{homologically pure}} if there is 
\begin{enumerate}
\item a set \(\mathcal I_E\), 
\item a function \(i\mapsto k_i\) assigning to elements of \(\mathcal I_E\) an integer, and 
\item a function \(i\mapsto H_i\) assigning to elements of \(\mathcal I_E\) a subgroup of \(G\),
\end{enumerate} 
such that we have an equivalence of \(H\mZ\)-modules
\[
H\mZ\wedge E\simeq H\mZ\wedge \bigvee_{i\in\mathcal I_E} G_{+}\smashover{H_i}S^{k_i\rho_{H_i}}.
\]

A homologically pure \(G\)-spectrum \(E\) is {\defemph{isotropic}} if there are no summands with a trivial stabilizer.
\end{definition}

\begin{remark}
A slightly restricted form of this definition was independently given by Pitsch--Ricka--Scherer in their analysis of conjugation spaces \cite{PRS}. The choice name and reason for the name are the same as the one here: analogy with \cite{HHR}.
\end{remark}

In fact, we can work more generally, using arbitrary zero-slices. Any zero-slice is of the form \(H\mM\) for some Mackey functor in which all restriction maps along surjective maps are injections \cite[Proposition 4.50]{HHR}, and the map \(\m{\pi}_0\) induces an equivalence between zero-slices and the full subcategory of Mackey functors of this form.

\begin{notation}
	We say that a Mackey functor \(\mM\) is a zero-slice if \(H\mM\) is zero-slice.
\end{notation}

Since the zero-slice of the zero sphere is \(H\mZ\), any zero-slice is a module over \(H\mZ\). This shows that we could have instead used arbitrary zero-slices. 

\begin{proposition}
A \(G\)-spectrum \(E\) is homologically pure if and only if for every zero-slice \(\mM\), we have an equivalence of \(\mZ\)-modules
\[
H\mM\wedge E\simeq H\mM\wedge \bigvee_{i\in\mathcal I_E} G_+\smashover{H_i} S^{k_i\rho_{H_i}}.
\]
\end{proposition}

\begin{notation}
If we have a decomposition like that of homological purity or isotropic homological purity only for particular Green zero-slices \(\mR\), then we will say that \(E\) has [isotropic] homological purity for \(\mR\).
\end{notation}

The regular representations are closed under restrictions, conjugations, and inductions. This gives the following.
\begin{proposition}
If \(E\) is a homologically pure \(H\)-spectrum, then 
\begin{enumerate}
\item \(G_{+}\smashover{H}E\) is homologically pure,
\item if \(K\subset H\), then \(i_{K}^{\ast}E\) is homologically pure, and
\item \(N_{H}^{G}E\) is homologically pure.
\end{enumerate}
\end{proposition}

\subsubsection{Homology}
The main benefit of this definition is from the defining property of zero-slices: all restriction maps are injections, and hence statements can usually be checked at the level of underlying homology.

\begin{notation}
Given an indexing set \(\mathcal I_{E}\) for a homologically pure \(E\), for each integer \(n\), let
\[
\mathcal I_{E}^{n}=\big\{i\in\mathcal I_{E}\mid k_i|H_i|=n\big\}.
\]
\end{notation}

\begin{proposition}\label{prop:HomologyofHomologicallyPure}
Let \(E\) be homologically pure and let \(\mM\) be a zero-slice. For any subgroup \(K\) and for any integer \(k\), we have
\[
\m{H}_{k\rho_K}(E;\mM)\cong \bigoplus_{j\in\mathcal I_{E}^{k|K|}} i_{K}^{\ast}\mM_{i_{K}^{\ast}G/H_j}.
\]
We also have
\[
\m{H}_{k\rho_{K}-1}(E;\mM)\cong \bigoplus_{j\in\mathcal I_{E}^{(k|K|-1)}} \bigoplus_{\underset{K\cap gH_jg^{-1}=\{e\}}{g\in K\backslash G/H_j}} i_{K}^{\ast}\mM_{K}.
\]

In particular, all restriction maps are injections.
\end{proposition}

\begin{proof}
By assumption, \(H\mM\wedge E\) is a wedge of induced up regular representation spheres smashed with \(H\mM\), and hence a wedge of slices. We therefore have 
\begin{align*}
&\m{\pi}_{k\rho_{K}}\big(i_{K}^{\ast}(H\mM\wedge E)\big)\cong \bigoplus_{j\in\mathcal I_{E}^{k|K|}} \m{\big[S^{k\rho_{K}},i_{K}^{\ast}(G_{+}\smashover{H_j}S^{k_j\rho_{H_j}}\wedge H\mM)\big]}\cong \\
&\bigoplus_{j\in\mathcal I_{E}^{k|K|}} \m{\big[S^{0},i_{K}^{\ast}(G/H_{j+}\wedge H\mM)\big]}
\end{align*}
The result follows by the definition of \(i_{K}^{\ast}\mM_{i_{K}^{\ast}G/H_j}\).

For the case of \((k\rho_{K}-1)\), the argument is identical until the last step. Here, we have a direct sum
\[
\bigoplus \m{\pi}_{-1}\big(S^{-k\rho_{K}}\wedge i_{K}^{\ast}(G_{+}\smashover{H_j}S^{k_j\rho_{H_j}}\wedge H\mM)\big)
\]
The double coset decomposition of \(G\) as a \((K,H_j)\)-biset allows us to rewrite each summand:
\begin{multline*}
\m{\pi}_{-1}\big(S^{-k\rho_{K}}\wedge i_{K}^{\ast}(G_{+}\smashover{H_j}S^{k_j\rho_{H_j}}\wedge H\mM)\big)\cong \\
\bigoplus_{g\in K\backslash G/H_j} \Ind_{(K\cap gH_jg^{-1})}^{K}\m{\pi}_{-1}\big(S^{(n-m)\rho_{(K\cap gH_jg^{-1})}}\wedge H\mM\big),
\end{multline*}
where 
\[
n=k_j[H_j:H_j\cap g^{-1}Kg]\text{ and }m=k[K:K\cap gH_jg^{-1}].
\]
The only regular representation sphere that has a non-trivial homology in degree \(-1\) is the one for the trivial group in degree \(-1\), which gives the second part.
\end{proof}

\begin{corollary}
If \(G=C_{p^{n}}\) and \(E\) is homologically pure and isotropic, then the homology groups in dimensions of the form \((i\rho_{H}-1)\) vanish.
\end{corollary}

\begin{definition}
A homologically pure \(G\)-spectrum \(E\) is {\defemph{generalized isotropic}} if there is no pair 
\[
i\in\mathcal I_{E}^{n}\text{ and }j\in \mathcal I_{E}^{n-1}
\]
such that \(G/H_i\times G/H_j\) contains a free summand.
\end{definition}

This generalized isotropic condition allows us to have other ways to check homological purity.

\begin{theorem}
Let \(E\) be a \(G\) spectrum that admits a filtration such that \(gr(E)\) is homologically pure and generalized isotropic. Then \(E\) is homologically pure and generalized isotropic.
\end{theorem}
\begin{proof}
The filtration on \(E\) gives a spectral sequence with \(E_{1}\)-term
\[
\pi_{\mstar}\big(gr(E)\wedge H\mZ\big)\cong H_{\mstar}\big(gr(E);\mZ\big).
\]
By assumption, this is a free \(H\mZ_{\mstar}\)-module, and the generators are in dimensions \(k_i\rho_{H_i}\) for \(i\in\mathcal I_{gr(E)}\). The generalized isotropic condition guarantees that these classes are permanent cycles, since there are no possible targets for the differentials on the generators by 
Proposition~\ref{prop:HomologyofHomologicallyPure}. Thus \(E_{1}=E_{\infty}\), and since this is a free module, there are no possible extensions. 
\end{proof}

The same proof applies more generally to deduce \(R\)-freeness for pure and isotropic \(R\).  More generally, we also deduce nice properties for the \(R\)-homology of a homologically pure \(E\) for any \(R\) which is pure, provided we have the same kind of generalized isotropy.

\begin{definition}
A \(G\)-spectrum \(R\) is {\defemph{weakly pure}} if 
for each \(n\in\mathbb Z\), there is a set \(\mathcal I_{n}\) and for each \(i\in\mathcal I_{n}\), a subgroup \(H_{i}\) and a zero-slice \(\mM_{i}\) for \(H_{i}\) such that
the regular \(n\)-slice of \(R\) is
\[
P_{n}^{n}R\simeq \bigvee_{i\in\mathcal I_{n}} G_{+}\smashover{H_{i}}\Big(S^{\tfrac{n}{|H|}\rho_{H}}\wedge H\mM_{i}\Big).
\]
\end{definition}

\begin{remark}
	If a \(G\)-spectrum \(R\) is weakly pure, then the regular slice filtration of \(R\) is the same as the classical slice filtration of \(R\). This is because for each \(n\), the fiber of \(P^nR\to P^{n-1}R\) is also a classical \(n\)-slice. By \cite[Proposition~4.45]{HHR}, this must be the classical slice tower.
\end{remark}

\begin{theorem}
Let \(R\) be a weakly pure \(G\)-ring spectrum that is slice \(0\)-connective, and let \(E\) be a homologically pure spectrum such that \(\mathcal I_E^n\) is empty for \(n\) sufficiently negative. 

If for each \(j\in \mathcal I_{E}^{n}\), there is no \(k\in\mathbb Z\), \(i_E\in\mathcal I_{E}^{n-1-k}\), and \(i_R\in\mathcal I_{k}\) such that 
\[
G/H_j\times G/H_{i_E}\times G/H_{i_R}
\] 
contains a trivial summand, then 
\[
R_{\mstar}E\cong \bigoplus_{j\in\mathcal I_{E}} R_{\mstar}\big(G_{+}\smashover{H_j}S^{k_j\rho_{H_j}}\big).
\]
\end{theorem}
\begin{proof}
The slice filtration of \(R\) gives a filtration on \(E\wedge R\) with associated graded
\[
\bigvee_{n\in\mathbb Z} \left(\bigvee_{i\in\mathcal I_{n}} G_{+}\smashover{H_{i}}S^{\tfrac{n}{|H_{i}|}\rho_{H_{i}}}\wedge H\mM_{n,i}\right)\wedge E.
\]
Since \(E\) is homologically pure, this is equivalent to
\[
\bigvee_{j\in\mathcal I_{E}}\bigvee_{n\in\mathbb Z} \left(\bigvee_{i\in\mathcal I_{n}} G_{+}\smashover{H_{i}}S^{\tfrac{n}{|H_{i}|}\rho_{H_{i}}}\wedge H\mM_{n,i}\wedge G_{+}\smashover{H_j}S^{k_j\rho_{H_j}}\right).
\]
The \(E_{2}\)-term of the associated spectral sequence is
\[
E_{2}(R\wedge E)\cong E_{2}(R)\Boxover{H_{\mstar}} H_{\mstar}(E;\mZ),
\]
which is a free module over the \(E_{2}\)-term for \(R\) with a basis given by a basis for the homology of \(E\). Our assumption on the lack of free summands guarantees that there are no possible targets for differentials on the basis elements, since the corresponding group of possible targets vanishes. 

This gives us a map of \(R\)-modules:
\[
R\wedge \bigvee_{j\in\mathcal I_{E}} G_{+}\smashover{H_j}S^{k_j\rho_{H_j}}\to R\wedge E.
\]
By construction, this induces a map of filtered spectra, and hence a map of spectral sequences. This map is an isomorphism on \(E_2\), which implies that the map is a weak equivalence, since our assumptions on \(R\) and \(E\) guarantee that slice spectral sequence converges strongly. 
\end{proof}

\subsubsection{Cohomology}
We can make similar statements about the cohomology.

\begin{proposition}
If \(E\) is homologically pure and \(|\mathcal I_E^{n}|<\infty\) for all \(n\), then for any zero-slice \(\mM\), we have an equivalence of \(H\mZ\)-modules
\[
F(E,H\mM)\simeq H\mM\wedge \bigvee_{i\in\mathcal I_E} G_+\smashover{H_i}S^{-k_i\rho_{H_i}}.
\]
\end{proposition}
\begin{proof}
Since zero-slices are \(H\mZ\)-modules, we have an equivalance of \(H\mZ\)-modules
\[
F(E,H\mM)\simeq F_{H\mZ}(H\mZ\wedge E,H\mM).
\]
The homological purity of \(E\) gives an equivalence of \(H\mZ\)-modules
\[
H\mZ\wedge E\simeq H\mZ\wedge\bigvee_{i\in\mathcal I_E} G_+\smashover{H_i}S^{k_i\rho_{H_i}},
\]
and hence we have
\[
F(E,H\mM)\simeq\prod_{n}\prod_{i\in\mathcal I_{E}^{n}} G_{+}\smashover{H_i}S^{-k_i\rho_{H_i}}\wedge H\mM.
\]
Since \(\mathcal I_{E}^{n}\) is finite, the inner most products are the same as wedges. Since for all integers \(k\) and subgroups \(H\), the homotopy Mackey functors of 
\[
G_{+}\smashover{H}S^{-k\rho_{H}}\wedge H\mM
\] 
are zero outside of a finite range (depending only on \(k\) and \(H\)), the outer most product is also equivalent to the wedge.
\end{proof}

\begin{example}
An  theorem of Pitsch--Ricka--Schrerer shows that any conjugation space of Hausman--Holm--Puppe \cite{HHP} are ``mod \(2\)'' homologically pure and isotropic \cite{PRS}. This gives a large class of examples.
\end{example}

\subsection{Consequences in computations}
The condition of homological purity gives surprising computational control.
\subsubsection{Green functor structure}
\begin{theorem}
Let \(E\) be a homologically pure spectrum, and assume that \(E\) comes equipped with a [commutative, associative] multiplication in the homotopy category. Then for any commutative Green functor \(\mR\) which is a zero-slice, the multiplication on 
\[
{H}_{\mstar}(E;\mR)
\]
is completely determined by the restrictions to 
\[
H_\ast(i_e^\ast E;\mR(G)).
\]
\end{theorem}
\begin{proof}
The homological purity of \(E\) guarantees that the homology and cohomology are free modules over the \(\RO\)-graded homology of a point. In particular, the ring structure is completely determined by the products of basis vectors. These occur in dimensions of the form \(k\rho_H\) for various \(k\) and \(H\). If \(x\in H_{k\rho_H}(E;\mZ)\) and \(y\in H_{\ell\rho_J}(E;\mZ)\), then the product of \(x\) and \(y\) is represented by a map out of
\[
(G_+\smashover{H}S^{k\rho_H})\wedge (G_+\smashover{J}S^{\ell\rho_J}).
\]
This is a wedge of spaces of the form 
\[
G_+\smashover{K}S^{m\rho_K},
\]
where \(K\) ranges over all subgroups of the form \(H\cap gJg^{-1}\) and where 
\[
m\rho_K=i_{K}^\ast (k\rho_H)+i_K^\ast c_g^\ast (\ell\rho_J).
\]
In particular, this is a wedge of regular slice spheres, again, and hence the product takes values in a zero-slice by Proposition~\ref{prop:HomologyofHomologicallyPure}. Since all restriction maps are injections here, the result follows.
\end{proof}

\begin{corollary}
If \(E\) is a homologically pure spectrum, then the \(\RO\)-graded ring structure on the cohomology of \(E\) with coefficients in any commutative Green zero-slice is functorially determined by the underlying cohomology ring.
\end{corollary}

\subsubsection{Tambara functor structure}
If, moreover, \(E\) is a \(G\)-\(E_\infty\)-ring spectrum, then we also have good control over norms.

\begin{theorem}
If \(E\) is a \(G\)-commutative monoid in the homotopy category and if \(E\) is homologically pure, then for any Tambara zero-slice \(\mR\), we have that the norms in 
\[
H_\mstar(E;\mR)
\]
are determined by the formula
\[
i_e^\ast N_H^G(x)=\prod_{\gamma\in G/H}\gamma(i_{e}^{\ast}x).
\]
\end{theorem}
\begin{proof}
The proof is the same as for the products. Here we use that the collection of regular representations is a sub-semi-Mackey functor of the representation ring.
\end{proof}

\begin{remark}
Tambara functors which are also zero-slices were independently studied by Nakaoka, who called these ``MRC'' Tambara functors, in his study of localizations of Tambara functors \cite{NakaokaInvert}.
\end{remark}

\subsubsection{CoTambara structure}
Again, all of the desired structure can be read out of the underlying homology. The conorm maps are detected as twisted coproducts. The proofs are identical.

\begin{theorem}
Let \(E\) be a homologically pure spectrum, and assume that \(E\) comes equipped with a [cocommutative, coassociative] comultiplication in the homotopy category. Then for any commutative Green functor \(\mR\) which is a zero-slice, the comultiplication on 
\[
{H}_{\mstar}(E;\mR)
\]
is completely determined by
\[
H_\ast\big(i_e^\ast E;\mR(G)\big).
\]
\end{theorem}

\begin{theorem}
If \(E\) is a \(G\)-co-commutative comonoid in the homotopy category and if \(E\) is homologically pure, then for any Tambara zero-slice \(\mR\), we have that the conorms in 
\[
H_\mstar(E;\mR)
\]
are determined by the formula
\[
i_e^\ast N_H^G(x)=\Big(\big(\bigotimes_{\gamma\in G/H}\gamma\big)\circ\psi\Big)(i_{e}^\ast x).
\]
\end{theorem}

\subsubsection{{\DL} operations}
Finally, we restrict to \(C_2\). None of the arguments here are that specific to \(C_2\); the only issue is in defining the appropriate {\DL} operations. For groups which contain \(C_2\), norm arguments provide analogous classes, but the author has no idea in general. We recall Wilson's \(\RO(C_{2})\)-graded stable operations.

\begin{theorem}[{\cite[\S 3]{BehWil}, \cite{WilsonDL}}]
For each \(i\geq 0\) and for each \(\epsilon=0,1\), we have {\DL} operations
\[
Q^{i\rho_{2}-\epsilon}\colon H_{\star}(\mhyphen;\m{\F}_{2})\to H_{\star+i\rho_{2}-\epsilon}(\mhyphen;\m{\F}_{2}).
\]
When \(\star=i\rho_{2}\), \(Q^{i\rho_{2}}\) is the square.
\end{theorem}

In this case, homological purity says that the underlying structure describes everything.

\begin{theorem}\label{thm:DLAction}
If \(E\) is a homologically pure \(C_2\)-\(E_\infty\)-ring spectrum, then we have
\[
Q^{i\rho_2-\epsilon}\colon H_{\star}(E;\m{\F}_2)\to H_{\star+i\rho_2-\epsilon}(E;\m{\F}_2)
\]
is determined by the restrictions \(i_e^\ast Q^{i\rho_2-\epsilon}\). 
The ``odd'' operations \(Q^{i\rho_{2}-1}\) can only land in cells induced from the trivial group.
\end{theorem}
\begin{proof}
This again follows immediately from the assumption of homological purity.
\end{proof}

\subsection{Example: the homology of \(\mathbb CP^{\infty}\) and of \(\BUR\)}\label{ssec:Examples}
\subsubsection{The \(C_{2}\)-\(E_{\infty}\) space \(\mathbb CP^{\infty}\)}
The standard cell structure for \(\mathbb CP^{\infty}\) has a unique cell in dimension \(k\rho_{2}\) for all \(k\geq 0\) and no other cells. In particular, it is homologically pure and isotropic, with a basis given by all
\[
\bar{b}_{n}\in H_{n\rho_{2}}(\mathbb CP^{\infty};\mZ)
\]
corresponding to the top cell of \(\mathbb CP^{n}\). The ring and coring structure then follows immediately from the underlying case.

\begin{proposition}
As a Green Hopf algebra, the homology of \(\mathbb CP^{\infty}\) with coefficients in \(\mZ\) is a divided power algebra on the primitive class \(\bar{b}_{1}\).
\end{proposition}

The norm (and conorm) maps are also determined by the underlying condition. Here we have
\[
i_{C_{2}}^{\ast}\big(N_{e}^{C_{2}}b_{n}\big)=-b_{n}^{2}=-\binom{2n}{n} b_{2n}.
\]

\begin{proposition}
The norms are given by
\[
N_{e}^{C_{2}} b_{n}=-\binom{2n}{n}\bar{b}_{2n}.
\]
\end{proposition}

The conorms are dual to the (negative) squaring operation.

\subsubsection{The homology of \(\BUR\)}
We begin with the computation of the homology of \(\BUR\) with coefficients in \(\mZ\). We give a slightly different proof than that of \cite{Kahn} and \cite{PitSch}, using instead our formulae above. This line of argument was undoubtedly known by Araki and Landweber.

\begin{theorem}\label{thm:AssociativeRingVersion}
There are classes 
\[
\bar{a}_{i}\in H_{i\rho_{2}}(\BUR;\mZ)
\]
such that the induced map on \(A_{\infty}\)-rings
\[
H\mZ\wedge\bigwedge_{i\geq 1} \mathbb S^0[\bar{a}_i]=H\mZ\wedge\mathbb S^{0}[\bar{a}_{1},\bar{a}_{2},\dots]\to 
H\mZ\wedge \BUR
\]
is an equivalence of \(C_{2}\)-equivariant associative algebras, and hence the \(C_{2}\)-space \(\BUR\) is homologically pure and isotropic.
\end{theorem}
\begin{proof}
Araki lifted the classical, non-equivariant description of \(\MU_{\ast}\MU\), showing 
\[
\MUR\wedge\MUR\simeq\MUR[\bar{a}_{1},\dots],
\]
and in particular, this is free with a basis in regular representation dimensions. The Thom isomorphism shows
\[
\MUR\wedge{\BUR}_{+}\simeq \MUR\wedge\MUR
\]
as \(C_{2}\)-\(E_{\infty}\) rings. Since \(H\mZ\) is a commutative ring spectrum under \(\MUR\), the result follows by base-change.
\end{proof}

\begin{remark}
The classical Schubert cell analysis works equally well here, and the underlying argument is essential the same as that of \cite{HHP}.
\end{remark}

\begin{corollary}
For any finite group \(G\) which contains \(C_{2}\), the coinduced \(G\)-space \(\Map^{C_{2}}(G,\BUR)\) is homologically pure and isotropic with basis given by the norm of the monomial basis.
\end{corollary}

\begin{notation}
Let
\[
{\mZ}_{\mstar}=\m\pi_{\star}H\mZ.
\]
\end{notation}

\begin{corollary}
We have an isomorphism of \(\RO(C_{2})\)-graded Green functors
\[
{H}_{\mstar}(\BUR;\mZ)\cong {\mZ}_{\mstar}[\bar{a}_{1},\dots],
\]
where \(|\bar{a}_{i}|=i\rho_{2}\).
\end{corollary}

We can also deduce the norms, coproducts, and conorms.

\begin{proposition}
The norms are given by
\[
N_{e}^{C_{2}}(a_{i})=(-1)^i\bar{a}_{i}^{2}.
\]
\end{proposition}

Finally, the co-Tambara structure is lifting the usual dual polynomial structure. Since the space \(\BUR\) is finite type, we can equivalently describe the cohomology ring and the norms there.

\begin{proposition}[{\cite{Kahn}}]
The cohomology ring of \(\BUR\) is
\[
{H}^{\mstar}(\BUR;\mZ)\cong\mZ_{\mstar}[\bar{c}_{1},\dots].
\]
Moreover, the inclusions of equivariant maximal tori into the \(U_{\mathbb R}(n)\) identify these Chern classes with the usual symmetric functions in the Chern roots.
\end{proposition}
\begin{proof}
Only the second part requires proof, since \(\BUR\) is homologically pure, isotropic, and of finite type. The same is true for the space \((\mathbb CP^{\infty})^{\times n}\). The induced map on cohomology is the determined by the underlying homology, and we reduce to the classical case.
\end{proof}

\begin{proposition}
The norms of the Chern classes are also the squares:
\[
N_{e}^{C_{2}}(c_{i})=(-1)^i\bar{c}_{i}^{2}.
\]
\end{proposition}

Finally, using Theorem~\ref{thm:DLAction}, we deduce the action of Wilson's {\DL} operations.

\begin{theorem}
The {\DL} operations \(Q^{i\rho_{2}}\) on \(H_{\star}(\BUR;\m{\F}_{2})\) act as
\[
Q^{i\rho_{2}}(\bar{a}_{j})=\binom{n}{r-n-1}\bar{a}_{i+j}\text{ mod decomposables}.
\]
The {\DL} operations \(Q^{i\rho_{2}-1}\) are identically zero.
\end{theorem}

\begin{proof}
Theorem~\ref{thm:DLAction} implies that these operations are completely determined by the underlying action. The ordinary {\DL} action on the homology of \(BU\) was determined by Kochman \cite{KoDL, Lance}.
\end{proof}

As an aside, this also gives the {\DL} action on the space \(BO\) by applying geometric fixed points.

\begin{corollary}[{\cite[Theorem 36]{KoDL}}]
In 
\[
H_{\star}(BO;\F_{2})\cong \F_{2}[e_{1},\dots],
\]
we have for all \(r\geq 0\) and \(n\geq 1\),
\[
Q^{r}(e_{n})=\binom{n}{r-n-1} e_{n+r}\text{ mod decomposables.}
\]
\end{corollary}

\section{Bar and twisted bar spectral sequences}
For \(R\)-free spectra, we have readily computable equivariant versions of the classical Rothenberg--Steenrod and Eilenberg--Moore spectral sequences. For \(G=C_{2}\), we also have twisted versions of these where the group acts also on the homotopy pullback diagram. We explain how these work here, giving an example for the bar spectral sequence.

\subsection{Bar and Rothenberg--Steenrod}
Let \(A\) be an associative monoid in \(G\)-spaces. Let \(X\) be a right \(A\)-space and let \(Y\) be a left \(A\)-space. In this case, the derived balanced product can be computed via the bar construction:
\[
X\otimesover{A}Y=B(X,A,Y),
\]
where \(B(X,A,Y)\) is the geometric realization of the simplicial complex
\[
k\mapsto B_{k}(X,A,Y)=X\times A^{\times k} \times Y,
\]
and where as usual, the structure maps are the actions or product in \(A\). 

\begin{remark}
Although the space underlying the (non-derived) version of \(X\otimesover{A} Y\) is just the ordinary \(X\timesover{A} Y\), we use the tensor product notation to stress the connection with the algebraic case and to distinguish from later pullback constructions.
\end{remark}

If \(A\) and either \(X\) or \(Y\) are \(R\)-free, then we have a bar spectral sequence computing the \(R\)-homology of \(X\otimesover{A} Y\).

\begin{theorem}
If \(A\) and either \(X\) or \(Y\) are \(R\)-free, then we have an Adams-graded spectral sequence
\[
E^{s,\mstar}_{2}=\Tor_{-s}^{R_{\mstar}(A)}
\big(R_{\mstar}(X),R_{\mstar}(Y)\big)\Rightarrow R_{\mstar-s}(X\otimesover{A} Y).
\]
\end{theorem}
\begin{proof}
Our assumptions guarantee that for each \(k\), the \(R\)-homology of \(B_{k}(X,A,Y)\) is given by
\[
R_{\mstar}\big(B_{k}(X,A,Y)\big)\cong R_{\mstar}(X)\boxover{R_{\mstar}} R_{\mstar}(A)^{\Box k}\boxover{R_{\mstar}} R_{\mstar}(Y),
\]
and the maps are the standard resolution computing \(\Tor\).
\end{proof}

\begin{remark}
If a basis for \(A\) and either \(X\) or \(Y\) can be chosen to be in \(\RO(G)\), then Lewis--Mandell give an \(\RO(G)\)-graded version of the K\"unneth spectral sequence which gives the exact same result. This is because our bar complex becomes the relative smash product upon taking \(\Sigma^{\infty}_{+}\). The resulting spectral sequence is the same \cite{LewisMandell}, since it is built the same way. 
\end{remark}

Applying cohomology instead to the bar construction when \(X=Y=\ast\) gives the Rothenberg--Steenrod spectral sequence \cite{RothSteen}. Our assumptions allow this to be determined as well.

\begin{theorem}
If \(A\) is \(R\)-free and \(A\) is finite type, then we have a spectral sequence
\[
E^{\ast,\mstar}_{2}=\Ext^{s}_{R^{\mstar}(A)}
\big(R^{\mstar},R^{\mstar}\big)\Rightarrow R^{\mstar-s}(BA).
\]
\end{theorem}

\subsubsection*{Example: \(B\BUR\)}
Since \(\BUR\) is \(H\mZ\)-free, we can run the bar spectral sequence to compute the homology of \(B\BUR\).

\begin{proposition}[{\cite{LewisMandell}}]
There is an Adams-style spectral sequence with 
\[
\m{E}^{2}_{s,\star}=\m{\Tor}^{-s}_{{H}_{\mstar}\BUR}({\mZ}_{\mstar},{\mZ}_{\mstar})\cong E_{{\mZ}_{\mstar}}\big(\bar{y}_{1},\dots\big)\Rightarrow H_{\mstar-s}\big(B\BUR;\mZ\big),
\]
where \(\bar{y}_{i}\) is the element in \(\Tor^{1}\) represented by \(\bar{a}_{i}\) and has bidegree \((-1,i\rho_{2})\).
\end{proposition}

Since all of the algebra generators are in filtration \((-1)\), this spectral sequence collapses at \(E_{2}\). This is a free \({\mZ}_{\mstar}\)-module, hence there are no additive extensions. There are, however, multiplicative extensions.

\begin{theorem}
As an \(\RO(C_{2})\)-graded Green functor, 
\[
{H}_{\mstar}(B\BUR;\mZ)\cong {\mZ}_{\mstar}[\bar{y}_{1},\bar{y}_{2},\dots]/ (\bar{y}_{i}^{2}-a_{\sigma}\bar{y}_{2i+1}),
\]
where \(\bar{y}_{i}\) is a fixed element of degree \(i\rho_{2}+1\).
\end{theorem}
\begin{proof}
The {\DL} operations commute with the homology suspension, and since this factors through the indecomposables, our earlier analysis gives on-the-nose identifications of the {\DL} actions. 

Wilson has shown that for a class in degree \((n\rho_{2}+1)\), the square is stable and can be written as
\[
(\mhyphen)^{2}=a_{\sigma}Q^{(n+1)\rho_{2}}+u_{\sigma}Q^{(n+1)\rho_{2}-1}.
\]
In particular, the squares are given by
\[
\bar{y}_{n}^{2}=a_{\sigma}Q^{(n+1)\rho_{2}}\bar{y}_{n}=a_{\sigma}\big[Q^{(n+1)\rho_{2}}\bar{a}_{n}\big]=a_{\sigma}[\bar{a}_{2n+1}]=a_{\sigma}\bar{y}_{2n+1}.\qedhere
\]
\end{proof}

\begin{remark}
The geometric fixed points of this are again polynomial, and we recover the result of Kochman on the homology of \(BBO\) \cite{KoDL}.
\end{remark}

\begin{remark}
The \(C_{2}\)-space \(\BUR\) is \(C_{2}\)-\(E_{\infty}\), so it makes sense to ask about norm maps here. The situation is more complicated. In fact, the \(\Tor\) term itself has a somewhat confusing relationship with the norms, since there is no reason for the homology suspension to set them equal to zero. Put another way, the usual argument shows that homology suspension factors through the ordinary module of K\"ahler differentials, but it will not necessarily factor through the module of genuine K\"ahler differentials of \cite{HillAQ}.
\end{remark}

%

Since the homology of \(B\BUR\) is free, we also get the homology of the coinduced \(B\BUR\).
\begin{theorem}
For any finite group \(G\) and inclusion \(C_2\subset G\), we have an isomorphism of \(\RO\)-graded Tambara functors
\[
H_{\mstar}\big(\Map^{C_{2}}(G,B\BUR);\mZ\big)\cong 
N_{C_{2}}^{G}\big(H_{\mstar}(B\BUR;\m\Z)\big).
\]
\end{theorem}

\subsection{Twisted bar spectral sequence}
In \(C_{2}\)-equivariant homotopy, we have an additional version of the \(E_{1}\)-operad: the \(E_{\sigma}\)-operad. Algebras for this have no multiplication on their fixed points, but they do have a transfer map and an underlying multiplication. A summary  can be found in \cite{SignedLoops}.

If \(A\) is an \(E_{\sigma}\)-algebra, then we can form a kind of balanced tensor product
\[
\begin{tikzcd}
{A}
	\ar[r]
	\ar[d]
	&
{X}
	\ar[d]
	\\
{X}
	\ar[r]
	&
{X\overset{\curvearrowleftright}{\otimesover{A}} X},
\end{tikzcd}
\]
where \(C_{2}\) acts on the whole diagram by swapping the two copies of \(X\). This amounts to the data of a space \(X\) acted on by the associative monoid \(i_{e}^{\ast}A\). The \(E_{\sigma}\)-structure on \(A\) means that the group action gives an isomorphism \(i_{e}^{\ast}A\cong i_{e}^{\ast}A^{op}\), and hence the action on \(X\) also canonically gives a right action. The twisted balanced product swaps the two factors of \(X\) and also then necessarily changes these left and right actions.

\begin{definition}
If \(A\) is an \(E_{\sigma}\)-algebra and \(X\) is an \(i_{e}^{\ast}A\)-module, then let
\[
B^{\sigma}\!(A;X)=B\big(A,\Map(C_{2},A),\Map(C_{2},X)\big),
\]
where the action of \(\Map(C_{2},A)\) on \(A\) is via the \(E_{\sigma}\)-structure.
\end{definition}

Perhaps the most interest case is when \(X\) is a point. In this case, work of Hahn--Shi and of Liu show that in this case \(B^\sigma\) is the appropriate ``signed de-looping'', providing a classifying space for \(E_\sigma\)-algebras \cite{HahnShi, Liu}.

\begin{theorem}
If \(A\) has \(R\)-free homology, then we have a spectral sequence which Adams indexed has the form
\[
E^{s,\mstar}_{2}=\Tor_{-s}^{N_{e}^{C_{2}}\big(i_{e}^{*}R_{\ast}(i_{e}^{\ast}A)\big)}
\Big(R_{\mstar}\big(\Map(C_{2},X)\big),R_{\mstar}(A)\Big)\Rightarrow R_{\mstar-s}\big(B^{\sigma}\!(A;X)\big).
\]
If \(X\) also has \(R\)-free homology, then the action of \( N_{e}^{C_{2}}\big(i_{e}^{\ast}R_{\ast}(i_{e}^{\ast}A)\big)\) on 
\[
R_{\mstar}\big(\Map(C_{2},X)\big)\cong N_{e}^{C_{2}}\big(i_{e}^{\ast}R_{\ast}(X)\big)
\]
is the one induced by functoriality.
\end{theorem}

\subsection{Eilenberg--Moore}

Following Rector, we build a geometric model of the Eilenberg--Moore spectral sequence \cite{Rector, SmithEM}. Just as non-equivariantly, any \(G\)-space is a coalgebra with comultiplication given by the diagonal map, and \(G\)-space \(X\) together with a map to a \(G\)-space \(B\) can be viewed as a \(B\)-comodule (and in fact, we have much more structure equivariantly coming from the twisted diagonals). This allows us to form the cosimplicial cobar complex as a model for the homotopy pullback.

If \(X\to B\) and \(B\leftarrow Y\) are maps of \(G\)-spaces, then a model for the homotopy pullback is given by
\[
X\timesover{B}^{h}Y\simeq coB(X,B,Y),
\]
where \(coB(X,B,Y)\) is the totalization of the cosimplicial complex
\[
k\mapsto X\times B^{\times k}\times Y,
\]
and where the structure maps are the diagonal of \(B\) or the respective coaction maps. If \(B\) and either \(X\) or \(Y\) are \(R\)-free and finite type, then we have a spectral sequence computing cohomology. In general, convergence of this spectral sequence is very delicate, just as classically. For this reason, we state the result only for Bredon homology with coefficients in a Green functor.

\begin{theorem}\label{thm:EMSS}
If \(B\) and either \(X\) or \(Y\) has \(\mR\)-free homology, then we have a spectral sequence
\[
E_{2}=\Tor_{-s}^{H^{\mstar}(B;\mR)}\big(H^{\mstar}(X;\mR),H^{\mstar}(Y;\mR)\big)\Rightarrow H^{\mstar+s}(X\timesover{B}^{h}Y;\mR).
\]
\end{theorem}

\subsection{Twisted Eilenberg--Moore}

Dual to the twisted pushout, we have a twisted homotopy pullback. 

\begin{definition}
If \(f\colon X\to i_{e}^{\ast}B\), then let
\(
X\overset{\curvearrowleftright}{\timesover{B}} X
\) be the defined by the homotopy pullback
\[
\begin{tikzcd}
{X\overset{\curvearrowleftright}{\timesover{B}} X}
	\ar[r]
	\ar[d]
	&
{\Map(C_{2},X)}
	\ar[d, "{\Map(C_{2},f)}"]
	\\
{B}
	\ar[r, "\Delta"']
	&
{\Map(C_{2},B).}
\end{tikzcd}
\]
\end{definition}

This is modeling a pullback diagram where now the group acts by swapping the two sides again. The homotopy pullback gives a version of the ordinary homotopy pullback where we replace the ordinary interval with the balanced interval \([-1,1]\) in the sign representation.


\begin{remark}
If \(X\) is a point, then this gives us the space of signed loops into \(B\).
%
\end{remark}

This pullback gives a cobar complex and hence an Eilenberg--Moore spectral sequence via Theorem~\ref{thm:EMSS}.

\begin{theorem}
If \(B\) has \(\mR\)-free homology, and if \(\mR\) is a Tambara functor then we have a spectral sequence
\[
E_{2}=\Tor_{-s}^{N_{e}^{C_{2}}H^{\ast}(i_{e}^{\ast}B;\mR(C_{2}))}\Big(H^{\mstar}\big(\Map(C_{2},X);\mR\big),H^{\mstar}(B;\mR)\Big)\Rightarrow H^{\mstar+s}\big(X\overset{\curvearrowleftright}{\timesover{B}}X;\mR\big).
\]
Moreover, if \(X\) also has \(\mR\)-free homology, then the action on 
\[
H^{\mstar}\big(\Map(C_{2},X);\mR\big)\cong N_{e}^{C_{2}} H^{\ast}(X;\mR(C_{2}))
\]
is induced by the non-equivariant one.
\end{theorem}

We believe that these spectral sequences will be useful in computing the cohomology of equivariant Eilenberg--Mac Lane spaces.

%

\bibliographystyle{plain}

\bibliography{Complete}

\end{document}